\documentclass{amsart}   	
\usepackage{geometry}                		
\usepackage{graphicx}				
\usepackage{amssymb}
\usepackage{amsmath}
\usepackage{amsfonts}
\usepackage{mathrsfs}
\usepackage{amsthm}
\usepackage{bm}
\usepackage{bbm}
\usepackage{color}
\usepackage{tikz}
\usepackage{tikz-cd}
\usepackage[all]{xy}
\usepackage{mathtools}

\numberwithin{equation}{section}

\def\an{\mathrm{an}}

\def\comp{\mathrm{comp}}

\def\cusp{\mathrm{cusp}}

\def\gr{\mathrm{gr}}

\def\rel{\mathrm{rel}}
\def\LCS{{\mathrm{LCS}}}
\def\SL{{\mathrm{SL}}}

\def\DR{{\mathrm{dR}}}
\def\B{\mathrm B}

\def\dbs{\backslash\hspace{-.04in}\backslash}
\def\e{\textbf{e}}
\def\h{\mathfrak h}
\def\kk{\Bbbk}
\def\li{\langle}
\def\ri{\rangle}

\def\u{\mathfrak u}
\def\bu{\boldsymbol{\u}}

\def\wt{\widetilde}

\def\A{\mathbb A}
\def\C{\mathbb C}

\def\cC{\check{C}}
\def\rmC{\check{\text{C}}}

\def\Eis{G}
\def\F{\mathcal F}
\def\cG{\mathcal G}
\def\G{\mathbb G}
\def\cH{\mathcal H}
\def\H{\mathbb H}
\def\Hbar{\overline{\mathcal H}}

\def\T{\mathsf{T}}
\def\Ss{\mathsf{S}}
\def\E{\mathcal E}

\def\LL{{\mathbb L}}

\def\M{\mathcal{M}}
\def\Mbar{\overline{\mathcal M}}
\def\OO{\mathcal O}

\def\Q{\mathbb Q}

\def\cR{\check{R}}
\def\cU{\mathcal U}
\def\bU{\boldsymbol{\cU}}
\def\U{\mathfrak U}
\def\V{\mathbb V}
\def\cV{\mathcal V}
\def\Vbar{\overline{\cV}}
\def\X{\mathsf{X}}
\def\Xbar{{\overline{X}}}
\def\Y{{\mathsf{Y}}}
\def\Ybar{{\overline{Y}}}
\def\Z{\mathbb Z}

\newtheorem*{theorem*}{Theorem}
\newtheorem{theorem}{Theorem}[section]
\newtheorem{proposition}[theorem]{Proposition}
\newtheorem{corollary}[theorem]{Corollary}
\newtheorem{lemma}[theorem]{Lemma}

\theoremstyle{definition}
\newtheorem{definition}[theorem]{Definition}
\newtheorem{example}[theorem]{Example}

\theoremstyle{remark}
\newtheorem{remark}[theorem]{Remark}

\newcommand\ms{\scriptscriptstyle}

\newcommand\Aut{\operatorname{Aut}}
\newcommand\Hom{\operatorname{Hom}}

\newcommand\Gr{\operatorname{Gr}}

\newcommand\Res{\operatorname{Res}}
\newcommand\Spec{\operatorname{Spec}}
\newcommand\Sym{\operatorname{Sym}}

\begin{document}

\title{Algebraic de Rham Theory for the Relative Completion of $\SL_2(\Z)$}
\author{Ma Luo}

\date{\today}							
\begin{abstract}
We develop an explicit algebraic de Rham theory for the relative completion of $\SL_2(\Z)$. This allows the construction of iterated integrals involving modular forms of the second kind, generalizing iterated integrals of holomorphic modular forms that were previously studied by Brown and Manin. These newly constructed iterated integrals provide all `multiple modular values' defined by Brown.
\end{abstract}

\maketitle

\tableofcontents

\section{Introduction}

Brown \cite{brown:mmm} defines multiple modular values to be periods of the coordinate ring $\OO(\cG^\rel)$ of the relative completion $\cG^\rel$ of $\SL_2(\Z)$. In this paper, we provide  an explicit algebraic de Rham theory for $\cG^\rel$, which enables us to explicitly construct iterated integrals of {\em modular forms of the second kind} (see below) that have at worst logarithmic singularities at the cusp. These newly constructed iterated integrals provide all multiple modular values, whereas previously only those `totally holomorphic' multiple modular values that are (regularized) iterated integrals of {\em holomorphic modular forms} (see below) have been studied by Brown \cite{brown:mmm} and Manin \cite{manin1,manin2}.

The relative completion $\cG^{\rel}_{/\Q}$ of $\text{SL}_2(\mathbb{Z})$ with respect to the inclusion $\rho\colon\text{SL}_2(\mathbb{Z})\xhookrightarrow{}\text{SL}_2(\mathbb{Q})$ is an extension of $\text{SL}_{2/\Q}$ by a prounipotent group $\cU^{\rel}_{/\Q}$. The Lie algebra $\u^{\rel}$ of $\cU^{\rel}$ is freely and topologically generated by \cite{hdr}
$$\prod_{n\ge 0} H^1(\text{SL}_2(\mathbb Z), S^{2n} H)^*\otimes S^{2n} H,$$
where $H$ is the standard representation of $\text{SL}_2$, and $S^{2n} H$ its $2n$-th symmetric power.  

The first step is to construct an explicit $\Q$-de Rham structure on $H^1(\SL_2(\Z), S^{2n} H)$. Recall that there is a mixed Hodge structure on $H^1(\SL_2(\Z), S^{2n} H)$, which has weight and Hodge filtrations defined over $\mathbb Q$ \cite{zucker}:
\begin{align*}
W_{2n+1}H^1(\SL_2(\Z), S^{2n} H)&=H^1_{\cusp}(\SL_2(\Z), S^{2n} H);\\
W_{4n+2}H^1(\SL_2(\Z), S^{2n} H)&=H^1(\SL_2(\Z), S^{2n} H);\\
F^{2n+1}H^1(\SL_2(\Z), S^{2n} H)&\cong\left\{
\begin{tabular}{c}
{\em holomorphic} part which consists of cohomology \\
classes that correspond to classical modular forms
\end{tabular}\right\}.
\end{align*}
Now we consider its $\Q$-de Rham structure $H^1_\DR(\M_{1,1/\Q}, S^{2n}\cH)$, where $\cH$ is the relative de Rham cohomology of the universal elliptic curve over the moduli space $\M_{1,1/\Q}$ of elliptic curves. In the holomorphic part $F^{2n+1}H^1(\SL_2(\Z), S^{2n} H)$, all $\Q$-de Rham classes correspond to classical modular forms of weight $2n+2$ with rational Fourier coefficients \cite[\S 21]{kzb}. We call these {\em holomorphic modular forms}. To obtain a complete $\Q$-de Rham basis of $H^1_\DR(\M_{1,1/\Q}, S^{2n}\cH)$, in Section \ref{nmf}, we use {\em modular forms of the second kind} as representatives for all $\Q$-de Rham classes. This differs from the traditional approach using weakly modular forms, which allows arbitrary poles at the cusp (cf. \cite{brown-hain, scholl-kazalicki}). Our representatives have at worst logarithmic singularities at the cusp. They are more suitable for constructing regularized iterated integrals (Section \ref{dbint}).
 
For each choice of a base point $x\in\M_{1,1}$, we identify $\SL_2(\Z)$ with the (orbifold) fundamental group $\pi_1(\M_{1,1}^\an,x)$. 
Denote the relative completion of $\pi_1(\M_{1,1}^\an,x)$ with respect to the inclusion $\rho_x:\SL_2(\Z)\xhookrightarrow{}\SL_2(\Q)$ by $\cG_x$. Denote the Lie algebra of its unipotent radical by $\u_x$. They are isomorphic to $\cG^\rel$ and $\u^\rel$ respectively, and they depend on the base point $x$. The Lie algebras $\u_x$ form a local system $\bu_\B$ over $\M_{1,1}^\an$. By a modification of Chen's method of power series connections (Prop. \ref{procedure}, cf. \cite{chen}), Hain \cite{hdr} finds a canonical universal flat connection $\nabla_\B$ on the Betti bundle $\bU\cong\bu_\B\otimes\OO_{\M_{1,1}^\an}$. We review this construction in Section \ref{Bstr} for a general affine curve (applicable to the case of $\M_{1,1}$). This provides, for each base point $x$, a natural Betti $\Q$-structure of the relative completion, which underlies a canonical mixed Hodge structure \cite{rc}. The connection $\nabla_\B$ is more general than the KZB connection in the elliptic curve case \cite{cee,kzb,mykzb}.

To provide a $\Q$-de Rham theory for $\cG^\rel$, we construct in Section \ref{coc} a canonical $\Q$-de Rham flat connection $\nabla_\DR$ on an algebraic de Rham vector bundle $\bu_\DR\to\M_{1,1}$. One naturally extends the Betti bundle $\bU$ and the de Rham bundle $\bu_\DR$ from $\M_{1,1}$ to $\Mbar_{1,1}$, and obtains $\overline\bU$ and $\overline\bu_\DR$ respectively. We will prove that
\begin{theorem*}[Theorem \ref{DRbdle}]
There is an isomorphism of bundles with connections 
$$(\overline\bu_\DR,\nabla_\DR)\otimes_\Q\C\approx(\overline\bU,\nabla_\B).$$
\end{theorem*}

 Moreover, we have that
\begin{theorem*}[Theorem \ref{QDRstr}]
The de Rham vector bundle $(\bu_\DR,\nabla_\DR)$ allows us to concretely construct a $\Q$-de Rham structure on the relative completion of $\SL_2(\Z)$ for each base point $x\in\M_{1,1}(\Q)$.
\end{theorem*}
To find the connection $\nabla_\DR$, we use a $\rmC$ech--de Rham version of Chen's method of power series connections (Prop. \ref{procedure}). Geometrically, we trivialize $\bu_\DR$ on a $\rmC$ech open cover of $\Mbar_{1,1}$ consisting of $\Mbar_{1,1}-\{[i]\}$ and $\Mbar_{1,1}-\{[\rho]\}$, with $\rho=e^{2\pi i/3}$. Chen's method provides an inductive algorithm for constructing the connections on both opens, and for finding the gauge transformation on their intersection. This geometric construction is reminiscent of the description for the elliptic KZB connection in \cite{mykzb}.

One of the main applications for our $\Q$-de Rham theory is the construction of all closed\footnote{i.e. homotopy invariant.} iterated integrals of {\em modular forms of the second kind}, which enables us to provide all multiple modular values. In the last part (Part \ref{pmf}), we illustrate with some explicit examples of non-totally holomorphic multiple modular values. In Section \ref{pcf}, we compute the quasi-periods of the Ramanujan cusp form $\Delta$ of weight 12, cf. \cite{brown:real,brown-hain}. In Section \ref{dbint}, we explicitly construct a double integral, i.e. iterated integral of length two, that involves a {\em non-holomorphic modular form}\footnote{This refers to a modular form of the second kind that is not a holomorphic modular form.}. This explicit construction is applicable to arbitrary lengths, from which one can compute all multiple modular values.  

\bigskip
\noindent{\em Acknowledgements:} This work was partially supported by ERC grant 724638. Most of this paper was written in Duke University as a part of my thesis. The paper was revised first in a visit to Hausdorff Research Institute for Mathematics, then at University of Oxford. I would like to thank my advisor, Richard Hain, for his support and suggestions in writing this paper. I would like to thank Francis Brown for sharing his algorithm and numerical results on quasi-periods. I would like to thank Erik Panzer for helping me with numerical computations. I would also like to thank the referees for careful readings, instructive suggestions and useful comments. 

\part{Background}

\section{The Moduli Space $\Mbar_{1,1}$ and its Open Cover}\label{basic}

In this section, we review the construction of the moduli stack of elliptic curves $\M_{1,1}$ and its Deligne--Mumford compactification $\Mbar_{1,1}$. Full details can be found in \cite{km}, complemented by \cite{olsson} for stacks.

\subsection{The moduli stack $\Mbar_{1,1}$ of stable elliptic curves}
Although the constructions below work for any field that is not of characteristic 2 or 3, we will work over $\Q$ exclusively. The moduli stack $\Mbar_{1,1/\Q}$ of stable elliptic curves is the stack quotient of 
$$\Ybar:=\mathbb A^2-\{(0,0)\}$$
by a $\G_m$-action
$$\lambda\cdot(u,v)=(\lambda^{4}u,\lambda^{6}v).$$ This $\G_m$-action is equivalent to a grading on the coordinate ring 
$$\OO(\Ybar):=\Q[u,v]=\bigoplus_d \gr_d \OO(\Ybar)$$
given by $\deg(u)=4$ and $\deg(v)=6$.

The discriminant function $$\Delta:=u^3-27v^2$$ has weight 12. Let $D$ be the discriminant locus defined by $\Delta=0$. The moduli stack $\M_{1,1}$ of elliptic curves is the stack quotient of 
$$Y:=\A^2-D$$
by the same $\G_m$-action above. 

\begin{remark}
The scheme $Y$ can be viewed as the moduli space of elliptic curves with a non-zero abelian differential. A point $(u,v)$ on $Y$ corresponds to the isomorphism class of the elliptic curve $y^2=4x^3-ux-v$ with an abelian differential $\frac{dx}{y}$. The $\G_m$-action acts on the abelian differential by $\lambda\cdot\frac{dx}{y}=\lambda^{-1}\frac{dx}{y}$, as it acts on the points of the elliptic curve by $\lambda\cdot (x,y)=(\lambda^2 x,\lambda^3 y)$.
\end{remark}

\begin{remark}
One could think of $\Mbar_{1,1}$ as a weighted projective space \cite{dolgachev}, which would make our later discussions on $\Mbar_{1,1}$ more natural. In the next subsection, we will define an ``affine" open cover of $\Mbar_{1,1}$ analogous to that of a projective space, and later use this open cover of $\Mbar_{1,1}$ to compute $\rmC$ech cohomology with twisted coefficients.
\end{remark}

\subsection{An open cover of $\Mbar_{1,1}$}\label{aoc}
Define $$Y_0:=\Spec\Q[u,v,u^{-1}],\quad Y_1:=\Spec\Q[u,v,v^{-1}].$$
Since $\Ybar=\mathbb A^2-\{(0,0)\}$, we have $\Ybar=Y_0\cup Y_1$. The $\G_m$-action on $\Ybar$: $\lambda\cdot(u,v)=(\lambda^{4}u,\lambda^{6}v)$, restricts to act on both $Y_0$ and $Y_1$. Note that for $i=0,1$, the coordinate rings $$\OO(Y_i)=\bigoplus_d\gr_d\OO(Y_i)$$ are graded with $u$, $v$, $u^{-1}$, $v^{-1}$ having weights $4$, $6$, $-4$, $-6$, respectively. Define 
\begin{equation}\label{opencover}
U_0:=\G_m\dbs Y_0,\quad U_1:=\G_m\dbs Y_1
\end{equation}
to be the stack quotients of the $\G_m$-action. Then $$\U=\{U_0,U_1\}$$ forms an open cover of $\Mbar_{1,1}$.

\begin{remark}
Analytically, it is easy to see that $U_0=\Mbar_{1,1}-\{[\rho]\}$ and $U_1=\Mbar_{1,1}-\{[i]\}$ with $\rho=e^{2\pi i/3}$. The cusp $C=[i\infty]\in\Mbar_{1,1}$, which corresponds to the isomorphism class of a nodal cubic, is in both $U_0$ and $U_1$.
\end{remark}

To develop a $\Q$-de Rham theory, we will use the stack descriptions (\ref{opencover}) of the open cover $\U=\{U_0,U_1\}$ of $\Mbar_{1,1}$, which means working $\G_m$-equivariantly on $Y_0$, $Y_1$ and $\Ybar$.

\section{Vector Bundles $S^{2n}\cH$ on $\M_{1,1}$ and their Canonical Extensions $S^{2n}\Hbar$ on $\Mbar_{1,1}$}

In this section, we follow the exposition in \cite[\S 2.6]{brown-hain} very closely, with notations adapted to this paper.
 
\subsection{The Gauss--Manin connection on $\cH$ over $\M_{1,1}$}\label{gaussmanin}
Define a trivial rank two vector bundle $\Hbar$ on $\Ybar$ by
$$\Hbar:=\OO_\Ybar\Ss\oplus\OO_\Ybar\T,$$
where the multiplicative group $\G_m$ acts on it by
$$\lambda\cdot\Ss=\lambda\Ss,\quad \lambda\cdot\T=\lambda^{-1}\T,$$
so $\Ss$ and $\T$ have weights $+1$ and $-1$ respectively. This vector bundle $\Hbar$ and its restriction $\cH$ to $Y$, descend to vector bundles $\Hbar$ over $\Mbar_{1,1}$ and $\cH$ over $\M_{1,1}$. 

Define the connection on $\Hbar$ and its symmetric powers
$$S^{2n}\Hbar:=\Sym^{2n}\Hbar=\bigoplus_{s+t=2n}\OO_\Ybar\Ss^s\T^t$$
by
\begin{equation}\label{nabla0}
\nabla_0 =d+\left(-\frac{1}{12}\frac{d\Delta}{\Delta}\T+\frac{3\alpha}{2\Delta}\Ss\right)\frac{\partial}{\partial \T}+\left(-\frac{u\alpha}{8\Delta}\T+\frac{1}{12}\frac{d\Delta}{\Delta}\Ss\right)\frac{\partial}{\partial \Ss},
\end{equation}
where
$\alpha=2udv-3vdu$, $\Delta=u^3-27v^2$ and $\frac{\partial}{\partial \Ss}, \frac{\partial}{\partial \T}$ a dual basis for $\Ss,\T$. This explicit formula was worked out in \cite[Prop. 19.6]{kzb}.\footnote{In \cite[\S 2.6]{brown-hain}, the same connection $\nabla_0$ is written in terms of 1-forms $\psi=\frac{1}{12}\frac{d\Delta}{\Delta}$ and $\omega=\frac{3\alpha}{2\Delta}$. } The connection is $\G_m$-invariant, and has regular singularities along the discriminant locus $D$. Note that when pulled back to each $\G_m$-orbit in $\Ybar$, this connection is the trivial connection $d$. Therefore, it descends to a connection on $\Hbar$, and its symmetric powers
$$S^{2n}\Hbar:=\Sym^{2n}\Hbar$$ over $\Mbar_{1,1}$. These bundles are the canonical extensions of $\cH$ and $S^{2n}\cH:=\Sym^{2n}\cH$ over $\M_{1,1}$.

\subsection{The local system $\H$ over $\M_{1,1}^\an$}

Let $\H:=R^1\pi_*\C$, where $\pi:\E^\an\to\M_{1,1}^\an$ is the universal elliptic curve. It is proved in \cite[Prop 5.2]{umem} that by the relative algebraic de Rham theorem, there is a natural isomorphism
\begin{equation}\label{reladr}
\cH\otimes_\Q \C\cong \H_\B\otimes_\Q \OO_{Y^\an},
\end{equation}
of bundles with connection over $Y^\an:=Y(\C)$, where $\H_\B:=R^1\pi_*\Q$, and $\cH$ provides a $\Q$-de Rham structure on $\H$. The bundle on the right hand side is endowed with the Gauss--Manin connection. Under the isomorphism, the sections $\Ss$ and $\T$ of $\cH$ over $Y$ correspond to de Rham classes represented by algebraic 1-forms $xdx/y$ and $dx/y$ respectively (cf. \cite[Prop. 2.1]{mykzb}).  

For each $n$, define the $2n$-th symmetric power
$S^{2n}\H:=\Sym^{2n}\H$
of $\H$ over $\M_{1,1}^\an$, then $S^{2n}\cH$ over $\M_{1,1}$ provides a $\Q$-de Rham structure.

\part{De Rham Theory for the Relative Completion of $\SL_2(\Z)$}\label{rcsl2}
\section{De Rham Cohomology with Twisted Coefficients}\label{ADR}

The relative completion $\cG^{\rel}$ of $\text{SL}_2(\mathbb{Z})$ with respect to the inclusion $\rho:\text{SL}_2(\mathbb{Z})\xhookrightarrow{}\text{SL}_2(\mathbb{Q})$, is an extension of $\text{SL}_2$ by a prounipotent group $\cU^{\rel}$. The Lie algebra $\u^{\rel}$ of $\cU^{\rel}$ is freely topologically generated by 
$$\prod_{n\ge 0} H^1(\text{SL}_2(\mathbb Z), S^{2n} H)^*\otimes S^{2n} H,$$
where $H$ is the standard representation of $\text{SL}_2$, and $S^{2n} H$ its $2n$-th symmetric power.

We identify $H^1(\text{SL}_2(\mathbb Z), S^{2n} H)$ with $H^1(\mathcal M_{1,1}^{\an}, S^{2n}\mathbb H_\B)$. To develop a $\Q$-de Rham theory for $\cG^{\rel}$, it is necessary to find a $\Q$-de Rham structure $H^1_\DR(\M_{1,1}, S^{2n}\cH)$ on $H^1(\mathcal M_{1,1}^{\an}, S^{2n}\mathbb H)$ first. 

It is known classically that one can describe the de Rham cohomology groups using ``differential forms of the second kind". In the context of modular forms, this has be done explicitly where the forms are known as {\it weakly modular forms} in the literature: see \cite{scholl-kazalicki}, and more recently \cite{brown-hain}. Our approach in this section will lead to yet another explicit description using what we call {\it modular forms of the second kind} (Section \ref{nmf}). The advantage is that our forms only have logarithmic singularities at the cusp while weakly modular forms can have arbitrary poles; this property lends itself well in constructing regularized iterated integrals (Section \ref{dbint}). The disadvantage is that we need to use a $\rmC$ech--de Rham complex which is natural but somewhat technical.

\subsection{Algebraic de Rham theorem}
Let $X$ be a smooth quasi-projective variety defined over $\kk$. Without loss of generality one can assume $X=\overline{X}-C$, where $\overline{X}$ is smooth projective, and $C$ is a normal crossing divisor in $\overline{X}$. Given a vector bundle $(\mathcal{V},\nabla)$ with flat connection over $X$, having regular singularities along $C$, denote by $\mathbb V$ the local system of horizontal sections of $\mathcal{V}^{\an}$ over $X^{\an}$. Define the twisted de Rham complex
$$\Omega_X^{\bullet}(\mathcal V):=\Omega_X^{\bullet}\otimes_{\OO_X}\mathcal V,$$
and denote its hypercohomology by
$H^{\bullet}_\DR(X,\mathcal{V}):=\mathbb H^{\bullet}(X,\Omega_X^{\bullet}(\mathcal V))$.
Deligne \cite[Cor. 6.3]{deligne:ode} proved the following version of the algebraic de Rham theorem for de Rham cohomology with twisted coefficients.
\begin{theorem}\label{DRstr}
 There is an isomorphism
\begin{equation}
H^{\bullet}_\DR(X,\mathcal{V})\otimes_\kk\mathbb C\cong H^{\bullet}(X^{\an},\mathbb V).
\end{equation}
\end{theorem}
\begin{remark}[Affine case]\label{affine}
When $X$ is affine, the de Rham structure $H^{\bullet}_\DR(X,\mathcal{V})$ can be computed as the cohomology $H^{\bullet}(\Gamma(X,\Omega_X^{\bullet}(\mathcal V)))$ of global sections of the twisted de Rham complex with differential given by the connection $\nabla$. 
\end{remark}
\begin{remark}\label{log}
One can replace $\Omega_X^{\bullet}$ by any complex that is quasi-isomorphic to $\Omega_X^{\bullet}$ or the direct image sheaf complex $i_*\Omega_X^{\bullet}$ (for the inclusion $i:X\xhookrightarrow{}\overline{X}$), for example the logarithmic de Rham complex $\Omega_{\overline{X}}^{\bullet}(\log C)$. Then the de Rham structure $H^{\bullet}_\DR(X,\mathcal{V})$ can be computed as the hypercohomology
$$\H^{\bullet}(\Xbar,\Omega_\Xbar^{\bullet}(\log C)\otimes_{\OO_{\Xbar}}\Vbar)$$
of the twisted logarithmic de Rham complex $$\Omega_\Xbar^{\bullet}(\log C)\otimes_{\OO_{\Xbar}}\Vbar,$$
where $\Vbar$ denotes the canonical extension of $\cV$ to $\Xbar$.
\end{remark}

\begin{example}{\bf $\Q$-de Rham structure $H^1_\DR(Y, S^{2n}\cH)$ on $H^1(Y^\an, S^{2n}\mathbb H)$.} Apply the above discussion to $\kk=\Q$, $X=Y$, $\Xbar=\Ybar$, and let $\cV$ be the bundle $S^{2n}\cH$ over $Y$ defined in Section \ref{gaussmanin}. Recall that we denoted by $\Vbar=S^{2n}\Hbar$ its canonical extension to $\Ybar$. Then the $\Q$-de Rham structure
$$H^1_\DR(X,\cV)=H^1_\DR(Y, S^{2n}\cH)$$
on $H^1(Y^\an, S^{2n}\H)$ can be computed by the hypercohomology
$$\H^1(\Ybar,\Omega^\bullet_{\Ybar}(\log D)\otimes S^{2n}\Hbar)$$
where $D$ is the discriminant locus defined in Section \ref{basic}, and the differential of the twisted de Rham complex is induced by the Gauss--Manin connection $\nabla_0$ on $\Hbar$ over $\Ybar$ given explicitly by (\ref{nabla0}) in Section \ref{gaussmanin}.
\end{example}

To find a $\Q$-de Rham structure $H^1_\DR(\M_{1,1}, S^{2n}\cH)$ on $H^1(\mathcal M_{1,1}^{\an}, S^{2n}\mathbb H)$, we recall the arguments in Brown--Hain \cite[\S 3]{brown-hain}. Since the differential $\nabla_0$ is $\G_m$-invariant, we obtain the subcomplex
$$\F_{2n}^\bullet:=(\Omega^\bullet_{\Ybar}(\log D)\otimes S^{2n}\Hbar)^{\G_m}$$ of $\G_m$-invariant forms on $\Ybar$, which descends to a complex $\F_{2n}^\bullet$ on $\Mbar_{1,1}$. Since $\G_m$ is connected, this also computes $H^1_\DR(Y,S^{2n}\cH)$. By computing the Leray spectral sequence for the $\G_m$-principal bundle $p:Y\to\M_{1,1}$, one gets natural isomorphisms
$$p^*:H^1_\DR(\M_{1,1},S^{2n}\cH)\xrightarrow{\simeq}H^1_\DR(Y,S^{2n}\cH)$$
for $n>0$. 

\begin{example}{\bf $\Q$-de Rham structure $H^1_\DR(\M_{1,1}, S^{2n}\cH)$ on $H^1(\mathcal M_{1,1}^\an, S^{2n}\mathbb H)$.}\label{prieg}
Assume $n>0$. Let  $X=\M_{1,1}$, $\Xbar=\Mbar_{1,1}$, and their coverings
$Y=\M_{1,\vec{1}}=\A^2-D$, $\Ybar=\A^2-\{(0,0)\}$. Let $C$ denote the cusp in $\Mbar_{1,1}$, then 
$$\Xbar=\G_m\dbs\Ybar,\quad X=\Xbar-C=\G_m\dbs (\Ybar-D)=\G_m\dbs Y,$$ where the multiplicative group $\G_m$ acts on $Y$ and $\Ybar$ as before in Section \ref{basic}.

Let $\cV$ and $\Vbar$ be the bundles $S^{2n}\cH$ and  $S^{2n}\Hbar$ over $\M_{1,1}$ defined in Section \ref{gaussmanin}. 
The $\Q$-de Rham structure
$$H^1_\DR(X,\cV)=H^1_\DR(\M_{1,1}, S^{2n}\cH)$$
on $H^1(\M_{1,1}^\an, S^{2n}\H)$, by what we just discussed, can be computed as the hypercohomology $\H^1(\Mbar_{1,1},\F_{2n}^\bullet)$ of the twisted logarithmic de Rham complex
$$
\F_{2n}^\bullet=(\Omega^\bullet_{\Ybar}(\log D)\otimes S^{2n}\Hbar)^{\G_m},
$$
whose differential is again induced by the Gauss--Manin connection $\nabla_0$.
\end{example}
\begin{remark}\label{sQstr}
As the complexes $\Omega^\bullet_\Ybar(\log D)$ and $\Omega^\bullet_Y$ are quasi-isomorphic, their induced $\Q$-de Rham structures, $H^1_\DR(\M_{1,1},S^{2n}\cH)$ by us and $H^1_\DR(\M_{1,1},\mathscr{V}_{2n})$ by Brown--Hain \cite{brown-hain}, are the same. In particular, by Thm 1.2 in \cite{brown-hain}, this $\Q$-de Rham structure is equipped with the action of Hecke operators induced from \cite{guerzhoy}.
\end{remark}

\subsection{Sections of sheaves $\F^p_{2n}$ in the twisted de Rham complex $\F^\bullet_{2n}$}

Here we prepare ourselves for explicit computations later by writing down sections of sheaves $\F^p_{2n}$ over open sets $U_0$, $U_1$ and their intersection $U_{01}$.

First, we compute global sections of sheaves $\Omega^p_{\Ybar}(\log D)$ in the logarithmic de Rham complex. By applying Deligne's criterion (\cite[\S 3.1]{deligne:h2}) for being a global section, we have
\begin{lemma}\label{lgcplx}
$$
\Gamma\big(\Ybar,\Omega^p_{\Ybar}(\log D)\big)=
\begin{cases}
\OO(\Ybar) & \quad\quad p=0,\\
\OO(\Ybar)\frac{\alpha}{\Delta}\bigoplus\OO(\Ybar)\frac{d\Delta}{\Delta} & \quad\quad p=1,\\
\OO(\Ybar)\frac{du\wedge dv}{\Delta}& \quad\quad p=2,\\
0 & \quad\quad \text{otherwise.}
\end{cases}
$$
\end{lemma}
\begin{proof}
In our case, a section $\varphi$ is in the logarithmic complex if and only if $\Delta\cdot\varphi$ and $\Delta\cdot d\varphi$ are holomorphic on $\Ybar$. The only interesting case is when $p=1$. Assume that $\varphi$ is a 1-form in the logarithmic complex, then $\phi:=\Delta\cdot\varphi$ is holomorphic, we can write it as
$$\phi=fdu+gdv\in\OO(\Ybar) du\oplus\OO(\Ybar) dv,$$
where $f,g$ are polynomials in $u$ and $v$ as $\OO(\Ybar)=\Q[u,v]$. Since
$$
d\varphi=d\left(\frac{\phi}{\Delta}\right)=\frac{d\phi}{\Delta}+\phi\wedge\frac{d\Delta}{\Delta^2},
$$
to make sure that $\Delta\cdot d\varphi$ is holomorphic, the 2-form
$$
\phi\wedge d\Delta=(fdu+gdv)\wedge(3u^2du-54vdv)=-(3u^2g+54vf)du\wedge dv
$$
has to be a multiple of $\Delta$, i.e. we need to have
$$
(u^3-27v^2)\big|(3u^2g+54vf).
$$
We call $(f,g)$ a solution if it satisfies the above condition. One can check that $(f,g)=(-3v,2u)$ and $(f,g)=(3u^2,-54v)$ are solutions. We will show that $\phi=fdu+gdv$ has to be an $\OO(\Ybar)$-linear combination of these two given solutions, i.e.
$$
-3vdu+2udv\ (=\alpha)\quad\text{and}\quad 3u^2du-54vdv\ (=d\Delta),
$$
so that
$$
\varphi=\frac{\phi}{\Delta}\in\OO(\Ybar)\frac{\alpha}{\Delta}\oplus\OO(\Ybar)\frac{d\Delta}{\Delta}.
$$

By substracting a suitable $\OO(\Ybar)$-linear combination of the given solutions from $(f,g)$, one can assume that $g$ is a constant. From $\Delta\big|(3u^2g+54vf)$, we see that $g=0$, and that $f$ is a multiple of $\Delta$. But 
$$(\Delta,0)=9v\cdot(-3v,2u)+\frac u3\cdot(3u^2,-54v)$$ 
is already an $\OO(\Ybar)$-linear combination of the given solutions, so we are done.
\end{proof}
%
Our next task is to apply $\rmC$ech--de Rham theory to compute $\Q$-de Rham representatives of $H^1_\DR(\M_{1,1},S^{2n}\cH)$. In preparation, we compute the $\G_m$-invariant sections on our open cover.

Since everything is graded by $\G_m$, denote the degree $n$ part of $\OO(\Ybar)$ by $\gr_n\OO(\Ybar)$. It is easy to deduce from Lemma \ref{lgcplx} that by weight computations, global sections of $\F_{2n}^p$ are
\begin{align*}
\F^p_{2n}(\Mbar_{1,1})&=\Gamma\big(\Ybar,(\Omega^p_{\Ybar}(\log D)\otimes S^{2n}\Hbar)^{\G_m}\big)
\\&=
\begin{cases}
\bigoplus\limits_{\ms s+t=2n}(\gr_{t-s}\OO(\Ybar))\Ss^s\T^t & \quad\quad p=0,\\
\big(\bigoplus\limits_{\ms s+t=2n}(\gr_{t-s+2}\OO(\Ybar))\frac{\alpha}{\Delta}\Ss^s\T^t\big)\bigoplus \big(\bigoplus\limits_{\ms s+t=2n}(\gr_{t-s}\OO(\Ybar))\frac{d\Delta}{\Delta}\Ss^s\T^t\big) & \quad\quad p=1,\\
\bigoplus\limits_{\ms s+t=2n} (\gr_{t-s+2}\OO(\Ybar))\frac{du\wedge dv}{\Delta}\Ss^s\T^t & \quad\quad p=2,\\
0 & \quad\quad \text{otherwise.}
\end{cases}
\end{align*}
Note that for $i=0,1$, we have $\OO_\Ybar(Y_i)=\OO({Y_i})$. They are graded and sections of $\F^p_{2n}$ on $U_i$ are
\begin{align*}
\F^p_{2n}(U_i)&=\Gamma\big(Y_i,(\Omega^p_{\Ybar}(\log D)\otimes S^{2n}\Hbar)^{\G_m}\big)
\\&=
\begin{cases}
\bigoplus\limits_{\ms s+t=2n}(\gr_{t-s}\OO({Y_i}))\Ss^s\T^t & \quad\quad p=0,\\
\big(\bigoplus\limits_{\ms s+t=2n}(\gr_{t-s+2}\OO({Y_i}))\frac{\alpha}{\Delta}\Ss^s\T^t\big)\bigoplus \big(\bigoplus\limits_{\ms s+t=2n}(\gr_{t-s}\OO({Y_i}))\frac{d\Delta}{\Delta}\Ss^s\T^t\big) & \quad\quad p=1,\\
\bigoplus\limits_{\ms s+t=2n}(\gr_{t-s+2}\OO({Y_i}))\frac{du\wedge dv}{\Delta}\Ss^s\T^t & \quad\quad p=2,\\
0 & \quad\quad \text{otherwise.}
\end{cases}
\end{align*}
Let $Y_{01}:=Y_0\cap Y_1=\Spec \Q[u,v,u^{-1},v^{-1}]$, then the coordinate ring $\OO({Y_{01}})=\Q[u,v,u^{-1},v^{-1}]$ is graded with $u$, $v$, $u^{-1}$, $v^{-1}$ having weights $4$, $6$, $-4$, $-6$, respectively.  Since the $\G_m$-action on $\Ybar$ restricts to $Y_{01}$, and the stack quotient of this action is $U_{01}:=\G_m\dbs Y_{01}$, we deduce similarly that
\begin{align*}
\F^p_{2n}(U_{01})&=\Gamma\big(Y_{01},(\Omega^p_{\Ybar}(\log D)\otimes S^{2n}\Hbar)^{\G_m}\big)
\\&=
\begin{cases}
\bigoplus\limits_{\ms s+t=2n}(\gr_{t-s}\OO({Y_{01}}))\Ss^s\T^t & \quad\quad p=0,\\
\big(\bigoplus\limits_{\ms s+t=2n}(\gr_{t-s+2}\OO({Y_{01}}))\frac{\alpha}{\Delta}\Ss^s\T^t\big)\bigoplus \big(\bigoplus\limits_{\ms s+t=2n}(\gr_{t-s}\OO({Y_{01}}))\frac{d\Delta}{\Delta}\Ss^s\T^t\big) & \quad\quad p=1,\\
\bigoplus\limits_{\ms s+t=2n} (\gr_{t-s+2}\OO({Y_{01}}))\frac{du\wedge dv}{\Delta}\Ss^s\T^t & \quad\quad p=2,\\
0 & \quad\quad \text{otherwise.}
\end{cases}
\end{align*}

\subsection{The $\rmC$ech--de Rham complex $\cC^\bullet(\U,\F_{2n}^\bullet)$}\label{cech}
In this subsection, we construct a $\rmC$ech--de Rham complex $\cC^\bullet(\U,\F_{2n}^\bullet)$ that computes the $\Q$-de Rham structure $H^1_\DR(\M_{1,1}, S^{2n}\cH)$ for the open cover $\U=\{U_0,U_1\}$:
$$
\begin{array}{|ccc}
\F_{2n}^2(U_0)\oplus\F_{2n}^2(U_1) & \xrightarrow{\delta} & \F_{2n}^2(U_{01})\\
\text{\tiny $(\nabla_0,\nabla_0)$}\uparrow & & \text{\tiny $-\nabla_0$}\uparrow \\
\F_{2n}^1(U_0)\oplus\F_{2n}^1(U_1) & \xrightarrow{\delta} & \F_{2n}^1(U_{01})\\
\text{\tiny $(\nabla_0,\nabla_0)$}\uparrow & & \text{\tiny $-\nabla_0$}\uparrow \\
\F_{2n}^0(U_0)\oplus\F_{2n}^0(U_1) & \xrightarrow{\delta} & \F_{2n}^0(U_{01})\\
\hline 
\end{array}
$$
where $U_{01}$ is the intersection of $U_0$ and $U_1$; the horizontal differential $\delta$ is the usual one for a $\rmC$ech complex, and the vertical differential is $\nabla_0$ in the twisted de Rham complex $\F_{2n}^\bullet$.
Let
$$D=\delta+(-1)^q\nabla_0$$
be the (total) differential of the single complex $s\cC^\bullet(\U,\F_{2n}^\bullet)$ associated to the $\rmC$ech--de Rham double complex $\cC^\bullet(\U,\F_{2n}^\bullet)$. 

Recall from Example \ref{prieg} that the $\Q$-structure $H^1_\DR(\M_{1,1},S^{2n}\cH)$ is computed by the hypercohomology $\H^1(\Mbar_{1,1},\F_{2n}^\bullet)$, where $\F_{2n}^\bullet=(\Omega^\bullet_{\Ybar}(\log D)\otimes S^{2n}\Hbar)^{\G_m}$. Since $\F_{2n}^p$ is coherent on $\Mbar_{1,1}$ (see \cite{steenbrink}, or \cite{dolgachev} for an algebraic proof), and the open cover $\mathfrak U=\{U_0,U_1\}$ is a ``good cover"\footnote{In the sense of Bott and Tu \cite[\S 8]{bott-tu} that the (augmented) columns are exact in the $\rmC$ech--de Rham complex.}, the $\rmC$ech cohomology calculated for this open cover $\U$ will give us the correct answer, and we have
$$
H^1_\DR(\M_{1,1},S^{2n}\cH)=\H^1(\Mbar_{1,1},\F_{2n}^\bullet)=H^1(s\cC^\bullet(\U,\F_{2n}^\bullet)).
$$
\begin{remark}\label{global}
One can compute the cohomology $H^1(\Mbar_{1,1},\F_{2n}^\bullet)$ of global sections of $\F_{2n}^\bullet$ (cf. Remark \ref{affine}). One finds that its dimension equals that of $M_{2n+2}$, the $\Q$ vector space spanned by \emph{holomorphic} modular forms of weight $2n+2$ with rational Fourier coefficients (Section \ref{cmf}). It does not equal to the dimension of $H^1(\M_{1,1}^{\an}, S^{2n}\H)$. This is expected since $\Mbar_{1,1}$ is projective, and one needs to use hypercohomology of $\Ybar$ to compute $H^1_\DR(\M_{1,1},S^{2n}\cH)$ if one wants the correct result.
\end{remark}

\section{Holomorphic Modular Forms and Modular Forms of the Second Kind}\label{cmfano}
In this section, we apply the algebraic de Rham theory previously developed, and find explicitly all $\Q$-de Rham classes in the $\Q$-de Rham structure $H^1_\DR(\M_{1,1},S^{2n}\cH)$ on $H^1(\M_{1,1}^{\an}, S^{2n}\H)$ for every positive $n$. These $\Q$-de Rham classes are closely related to \emph{holomorphic} modular forms with rational Fourier coefficients. 

By Eichler--Shimura \cite{eichler,shimura} isomorphism and Zucker \cite{zucker}, after tensoring $H^1_\DR(\M_{1,1},S^{2n}\cH)$ with $\C$, one obtains a natural mixed Hodge structure on $H^1(\M_{1,1}^{\an}, S^{2n}\H)$, which has weight and Hodge filtrations defined over $\Q$:
\begin{align}
W_{2n+1}H^1(\M_{1,1}^{\an}, S^{2n}\H)&=H^1_{\cusp}(\M_{1,1}^{\an}, S^{2n}\H);\label{a}\\
W_{4n+2}H^1(\M_{1,1}^{\an}, S^{2n}\H)&=H^1(\M_{1,1}^{\an}, S^{2n}\H),\label{b}\\
F^{2n+1}H^1(\M_{1,1}^{\an}, S^{2n}\H)&\cong M_{2n+2}\otimes_\Q\C,\label{c}
\end{align}
where $M_{2n}$ denotes the $\Q$-vector space spanned by \emph{holomorphic} modular forms of weight $2n$ with rational Fourier coefficients.

The $\Q$-structure for the last part $F^{2n+1}H^1(\mathcal M_{1,1}^{\an}, S^{2n}\mathbb H)$ -- \emph{holomorphic} modular forms -- is well known. Every $\mathbb Q$-de Rham cohomology class in $F^{2n+1}H^1(\mathcal M_{1,1}^{\an}, S^{2n}\mathbb H)$ can be represented by a {\em global} $\mathbb G_m$-invariant $\nabla_0$-closed 1-form on $Y$ with coefficients in $S^{2n}\mathcal H$. The explicit correspondence has been found in \cite[\S 21]{kzb}, we record it here in our notation in the following section \ref{cmf}. 

The remaining $\Q$-de Rham classes are the ones, under the Eichler--Shimura isomorphism, that involve anti-holomorphic cusp forms. 
In Section \ref{nmf}, we explain how to find and represent all $\Q$-de Rham classes including these remaining classes using $\rmC$ech cocycles in the $\rmC$ech--de Rham complex $\cC^\bullet(\U,\F_{2n}^\bullet)$. The representatives will be called {\it modular forms of the second kind}. Similar results were obtained by Scholl--Kazalicki \cite{scholl-kazalicki} and Brown--Hain \cite{brown-hain} using weakly modular forms as representatives. Our representatives have the advantage of having logarithmic singularities at the cusp. They are better suited to constructing regularized iterated integrals (Section \ref{dbint}).

\subsection{Holomorphic modular forms}\label{cmf}
Analytically, define a map 
\begin{equation}\label{psi}
\psi: \h\to Y^\an
\end{equation} 
by 
$$\qquad\tau\mapsto (u=g_2(\tau),v=g_3(\tau))$$
where $g_2(\tau)=20\Eis_4(\tau)$, $g_3(\tau)=\frac 73\Eis_6(\tau)$, using Zagier's normalization for Eisenstein series and his notation $\Eis_4(\tau)$, $\Eis_6(\tau)$ \cite{zagier}. This map induces an isomorphism
$$\SL_2(\Z)\dbs\h\to\G_m\dbs Y^\an=\M_{1,1}^\an.$$

Given a \emph{holomorphic} modular form $f(\tau)$ of weight $2n+2$ with rational Fourier coefficients, it can be written as a polynomial $h$ in $g_2(\tau)$ and $g_3(\tau)$. From Hain \cite[\S 21]{kzb}, the cohomology class corresponding to $f(\tau)$ can be represented by a {\em global} 1-form
\begin{equation}\label{poly}
h(u,v)\frac{\alpha}{\Delta}\T^{2n}=h(u,v)\frac{2udv-3vdu}{u^3-27v^2}\T^{2n}\in\left(\Omega^1_{\Ybar}(\log D)\otimes F^{2n}S^{2n}\cH\right)^{\mathbb G_m}.
\end{equation}
 The pullback of this form along $\psi$ is, up to a rational multiple ($\frac 23$ to be precise, as the pullback of $\frac{\alpha}{\Delta}$ is $\frac 23\frac{dq}{q}$),
\begin{equation}\label{omegaf}
 h(g_2(\tau),g_3(\tau))\T^{2n}\frac{dq}{q}=2\pi i f(\tau)\T^{2n}d\tau, \quad \text{with }q=e^{2\pi i \tau}.
\end{equation}
We will denote the 1-form on $Y$ in (\ref{poly}) by $\omega_f$, then the 1-form on the upper half plane $\h$ in (\ref{omegaf}) is actually $\frac 32\psi^*\omega_f$. It will be used as the integrand of Eichler integral to compute periods, cf. (\ref{ei}).\footnote{The 1-form in (\ref{omegaf}) was denoted by $\omega_f$ in Hain--Matsumoto \cite[\S 9.1]{umem}, and by $\underline{f}(\tau)$ in Brown \cite[\S 2.1]{brown:mmm} to compute periods. }

\begin{example}\label{deltaform}
The Ramanujan cusp form $\Delta=q-24q^2+252q^3+\cdots$ of weight 12 corresponds to a rational polynomial $\Delta=u^3-27v^2$. So its corresponding class is represented by the 1-form
$$\omega_\Delta=\Delta\frac{2udv-3vdu}{u^3-27v^2}\T^{10}=(2udv-3vdu)\T^{10}.$$
\end{example}

\subsection{Modular forms of the second kind}\label{nmf}
In this subsection, we will find all $\Q$-de Rham classes in $H^1_\DR(\M_{1,1},S^{2n}\cH)$ as promised. Their representatives are {\it modular forms of the second kind}.
\begin{definition}\label{mfsk}
{\it Modular forms of the second kind} are 1-cocycles in the single complex $s\cC^\bullet(\U,\F_{2n}^\bullet)$ associated to the $\rmC$ech--de Rham double complex $\cC^\bullet(\U,\F_{2n}^\bullet)$. 
\end{definition}

We start by discussing what this definition means. In the $\rmC$ech--de Rham complex $\cC^\bullet(\U,\F_{2n}^\bullet)$, every 1-cochain $\wt{\omega}$ is of the form
$$\wt{\omega}=\begin{array}{|cc}
(\omega^{(0)},\omega^{(1)}) &0\\
0& l\\
\hline
\end{array}
$$
where $\omega^{(i)}\in\F_{2n}^1(U_i)$, and $l\in\F_{2n}^0(U_{01})$, so that $(\omega^{(0)},\omega^{(1)})\in\check{C}^0(\U,\F_{2n}^1)$ and $l\in\cC^1(\U,\F_{2n}^0)$. We often simply write it as
$\wt\omega=( \omega^{(0)},\omega^{(1)}; l)$ in this section.

A 1-cochain $\wt\omega=( \omega^{(0)},\omega^{(1)}; l)$ in $s\cC^\bullet(\U,\F_{2n}^\bullet)$ is a 1-cocycle whenever $D\wt\omega=0$; in other words, it is a cocycle if and only if $\nabla_0\omega^{(0)}=\nabla_0\omega^{(1)}=0$ and $$\delta(\omega^{(0)},\omega^{(1)})-\nabla_0 l=\omega^{(1)}-\omega^{(0)}-\nabla_0 l=0.$$

\begin{remark}\label{ginter}
One can view a modular form of the second kind $\wt\omega=( \omega^{(0)},\omega^{(1)}; l)$ geometrically as follows: it consists of closed 1-forms $\omega^{(0)}$ and $\omega^{(1)}$ on opens $U_0$ and $U_1$ respectively, and a function $l$ on the intersection $U_{01}$ that can patch these two 1-forms together (cf. Section \ref{inter}). 
\end{remark}

\begin{example}{\bf Holomorphic modular forms as modular forms of the second kind.}\label{cocycle}
As was shown in the last section, a \emph{holomorphic} modular form $f$ with rational Fourier coefficients of weight $2n+2$ gives rise to a cohomology class $[\omega_f]\in H^1_\DR(\M_{1,1},S^{2n}\cH)$, which can be represented by a {\em global} closed form $\omega_f$. Denote by $\omega_f^{(i)}$ the restriction of $\omega_f$ to $U_i$, $i=0,1$, and define
$$\wt\omega_f:=(\omega_f^{(0)},\omega_f^{(1)};0).$$
Then $\wt\omega_f$ is a 1-cocycle in our complex $s\cC^\bullet(\U,\F_{2n}^\bullet)$, i.e. a modular form of the second kind.
\end{example}

To find representatives for all $\Q$-de Rham classes in $H^1_\DR(\M_{1,1}, S^{2n}\cH)$, we use the second spectral sequence of the double complex $\cC^\bullet(\U,\F_{2n}^\bullet)$, with the second page $E_2$ obtained by successively computing homologies under the horizontal and vertical differentials 
$$E_2^{p,q}=H^q_{\text{\tiny $\nabla_0$}}(H^p_{\delta}(\cC^\bullet(\U,\F_{2n}^\bullet))).$$ 
Since the double complex $\cC^\bullet(\U,\F_{2n}^\bullet)$ concentrates in the first two column, this spectral sequence degenerates at $E_3$ trivially, i.e. $E_3=E_4=\cdots=E_\infty$.

We now use a spectral sequence zig-zag argument. Starting from $l$ in the lower right corner, it extends to a 1-cocycle $\wt\omega=( \omega^{(0)},\omega^{(1)}; l)$ if it represents a cohomology class in $E_3$:
$$
\begin{array}{|c}
\begin{tikzcd}
(0,0) & \\
(\omega^{(0)},\omega^{(1)}) \arrow{r}{\delta} \arrow{u}{\text{\tiny $(\nabla_0,\nabla_0)$}} & \phantom{0}\\
 & l \arrow{u}{\text{\tiny $-\nabla_0$}}
\end{tikzcd}\\
\hline
\end{array}
$$
i.e. $\nabla_0\omega^{(0)}=\nabla_0\omega^{(1)}=0$ and $\omega^{(1)}-\omega^{(0)}-\nabla_0 l=0$. By the degeneracy at $E_3$ of the spectral sequence, this $\wt\omega$ represents a cohomology class in $s\cC^\bullet(\U,\F_{2n}^\bullet)$ and the condition above is equivalent to it being closed: $D\wt\omega=0$. 

Given a trivial class $0\in H^1_\delta(\U,\F_{2n}^0)$, we can always choose $l=0$ to represent it. Then by the conditions above, we would have a {\em global} closed form $\omega$, which brings us back to the previous example. In fact, when $n<5$, the corresponding group $H^1_\delta(\U,\F_{2n}^0)$ is trivial. The first interesting case occurs at $n=5$, which is expected, since $H^1_\DR(\M_{1,1}, S^{10}\cH)$ corresponds to modular forms of weight $2n+2=12$, where a cusp form appears for the first time. 

\begin{example}{\bf First new $\Q$-de Rham class in $H^1_\DR(\M_{1,1}, S^{2n}\cH)$.}\label{fndrc}
When $n=5$, one can find a nonzero class in $H^1_\delta(\U,\F_{10}^0)$, cokernel of $\delta$ on the $0$-th row, represented by $\frac{1}{uv}\Ss^{10}\in\F_{10}^0(U_{01})$. The reason it is non-trivial is that $\frac{1}{uv}=u^{-1}v^{-1}$, having negative powers for both $u$ and $v$, cannot be the difference of two elements in $\Q[u,v][u^{-1}]$ and $\Q[u,v][v^{-1}]$ respectively, where every monomial has a nonnegative power of at least one variable.

To extend the nontrivial class $[\frac{1}{uv}\Ss^{10}]\in E_1^{1,0}$ through $E_2$, one needs to find $(\omega_{1,1}^{(0)},\omega_{1,1}^{(1)})$ such that 
$$\omega_{1,1}^{(1)}-\omega_{1,1}^{(0)}=\nabla_0 l_{1,1},$$
where $l_{1,1}$ represents the class $[\frac{1}{uv}\Ss^{10}]$. One can choose $l_{1,1}$ to be $\frac{1}{uv}\Ss^{10}$ in this case since
$$\nabla_0\left(\frac{1}{uv}\Ss^{10}\right)=-\left(\frac{9\alpha}{u^2\Delta}+\frac{u\alpha}{2v^2\Delta}\right)\Ss^{10}-\frac{5\alpha}{4v\Delta}\Ss^9\T$$
can be written as $\omega_{1,1}^{(1)}-\omega_{1,1}^{(0)}$ with $$\omega_{1,1}^{(0)}=\frac{9\alpha}{u^2\Delta}\Ss^{10}\in\F_{10}^1(U_0),$$ and $$\omega_{1,1}^{(1)}=-\frac{u\alpha}{2v^2\Delta}\Ss^{10}-\frac{5\alpha}{4v\Delta}\Ss^9\T\in\F_{10}^1(U_1).$$ One can easily compute and find that these two forms are $\nabla_0$-closed, so the cochain
$$\wt{\omega}_{1,1}=\begin{array}{|cc}
(\omega_{1,1}^{(0)},\omega_{1,1}^{(1)}) &0\\
0& l_{1,1}\\
\hline
\end{array}
$$
lives to $E_3$. Therefore, this cochain $\wt\omega_{1,1}$ is a modular form of the second kind and represents a class $[\wt\omega_{1,1}]$ in $H^1_\DR(\M_{1,1}, S^{10}\cH)$.
\end{example}

Using the same argument at the beginning of the above example, one easily gets
\begin{lemma}\label{span}
The cohomology group $H^1_\delta(\U,\F_{2n}^0)$ has a basis consists of classes $[\frac{1}{u^pv^q}\Ss^s\T^t]$ with positive integers $p$, $q$ such that $4p+6q\leq 2n$, and nonnegative integers $s=n+2p+3q$, $t=n-2p-3q$. 
\end{lemma}
\begin{remark}
Here $s$ and $t$ can be solved from $s+t=2n$ ($\Ss^s\T^t\in S^{2n}\cH$) and $s-t=4p+6q$ ($\frac{1}{u^pv^q}\Ss^s\T^t$ is $\G_m$-invariant, i.e. $\frac{1}{u^pv^q}\in\gr_{t-s}\OO({Y_{01}})$). 
\end{remark}
 
Since we already found all the \emph{holomorphic} modular forms in $H^1_\DR(\M_{1,1}, S^{2n}\cH)$, and based on Eichler--Shimura (\ref{a})--(\ref{c}), the remaining classes span a vector space of dimension equal to the one spanned by (anti)holomorphic cusp forms. Therefore, by Lemma \ref{dcl} below, if we were to find independent classes $[\wt\omega_{j,k}]$ in $H^1_\DR(\M_{1,1}, S^{2n}\cH)/F^{2n+1}$ for each positive integer pair $(j,k)$ with $4j+6k=2n$, we will have found all the remaining classes.

\begin{lemma}\label{dcl}
The number of positive integer pairs $(j,k)$ satisfying $4j+6k=2n$ is the same as the dimension of the space of cusp forms of weight $2n+2$.
\end{lemma}
\begin{proof}
The dimension of cusp forms of weight $2n+2$ is the number of normalized Hecke eigen cusp forms of the same weight. We can choose a basis to be $\Delta\cdot u^av^b$ where $a$ and $b$ are nonnegative integers, $u$ and $v$ are normalized Hecke eigenform of weight 4 and 6 respectively as usual. It suffices to show a bijection between the set of integer pairs $(j,k)$ and the set of integer pairs $(a,b)$. Clearly the weight of modular forms provides us with the restriction $4a+6b+12=2n+2$, or equivalently $4(a+1)+6(b+1)=2n$. Therefore, $j=a+1$ and $k=b+1$ gives us the bijection we need.
\end{proof}

The following result explains how to find representatives of all the remaining $\Q$-de Rham classes. 
\begin{proposition}\label{ncf}
For any pair of positive integers $j,k$ such that $4j+6k=2n$, there is a class $[\wt\omega_{j,k}]$ in $H^1_\DR(\M_{1,1}, S^{2n}\cH)$ represented by a $\rmC$ech 1-cocycle
$$\wt{\omega}_{j,k}=\begin{array}{|cc}
(\omega_{j,k}^{(0)},\omega_{j,k}^{(1)}) &0\\
0& l_{j,k}\\
\hline
\end{array}
$$
where
$$l_{j,k}=\sum_{s+t=2n}x_{s,t}\Ss^s\T^t\in\F_{2n}^0(U_{01}),$$
with $x_{s,t}\in\gr_{t-s}\OO({Y_{01}})$ and $x_{2n,0}=\frac{1}{u^jv^k}$. In other words, we can choose $l_{j,k}$ to be a $\Q$-linear combination of terms $\frac{1}{u^pv^q}\Ss^s\T^t$ starting with a term $\frac{1}{u^jv^k}\Ss^{2n}$ so that $\nabla_0 l_{j,k}$ can be expressed as $\omega_{j,k}^{(1)}-\omega_{j,k}^{(0)}$, with both $\omega_{j,k}^{(0)}$ and $\omega_{j,k}^{(1)}$ being $\nabla_0$-closed.
\end{proposition}
\begin{proof}
After familiarizing the concepts of {\em order} and {\em bad terms} in the next paragraph, the reader is advised to read Example \ref{msk} first, as the proof that follows uses the same strategy.

Let us order the terms in $l_{j,k}$ and $\nabla_0 l_{j,k}$ by the power of $\Ss$, then $\Ss^{2n}$ has the highest order $2n$, while $\T^{2n}$ has the lowest order $0$. By using (\ref{nabla0}), one can compute that
\begin{align}
\nabla_0\left(\frac{1}{u^pv^q}\Ss^s\T^t\right)=&\quad\frac{3t}{2u^pv^q}\frac{\alpha}{\Delta}\Ss^{s+1}\T^{t-1}\label{1}
\\&-\left(\frac{9p}{u^{p+1}v^{q-1}}\frac{\alpha}{\Delta}+\frac{q}{2u^{p-2}v^{q+1}}\frac{\alpha}{\Delta}\right)\Ss^s\T^t\label{2}
\\&\quad-\frac{s}{8u^{p-1}v^q}\frac{\alpha}{\Delta}\Ss^{s-1}\T^{t+1}\label{3}
\end{align}
where the terms on the right hand side in (\ref{1}), (\ref{2}), (\ref{3}) have orders $s+1$, $s$, $s-1$ respectively. Our objective is to express $\nabla_0 l_{j,k}$ as a difference $\omega_{j,k}^{(1)}-\omega_{j,k}^{(0)}$. We call a term ``bad" if its denominater has positive powers of both $u$ and $v$ (as we cannot write it as $\omega^{(1)}-\omega^{(0)}$). We will eliminate all bad terms appearing in $\nabla_0 l_{j,k}$ by adding $\nabla_0$-coboundaries. 

To find $l_{j,k}$, we start with $\frac{1}{u^jv^k}\Ss^{2n}$, then $\nabla_0(\frac{1}{u^jv^k}\Ss^{2n})$ has terms of order $2n$ and $2n-1$ (the $(2n+1)$-order term being $0$). We can use order $2n-1$ terms $x_{2n-1,1}\Ss^{2n-1}\T$ to cancel the order $2n$ bad terms in $\nabla_0(\frac{1}{u^jv^k}\Ss^{2n})$ since the order $2n$ term in $\nabla_0(x_{2n-1,1}\Ss^{2n-1}\T)$ is just a rational multiple of $x_{2n-1,1}\frac{\alpha}{\Delta}\Ss^{2n}$ by the formula in (\ref{1}). Now $\nabla_0(\frac{1}{u^jv^k}\Ss^{2n}+x_{2n-1,1}\Ss^{2n-1}\T)$ has bad terms of order at most $2n-1$. 

We can repeat this process until the bad terms have order $n+7$, due to the fact that there always exist positive integer solutions for $4p+6q=2r$ with $2r>12$, which correspond to bad terms of the form $\frac{1}{u^pv^q}\Ss^{n+r}\T^{n-r}$. They are used to cancel rational multiples of $\frac{1}{u^pv^q}\frac{\alpha}{\Delta}\Ss^{n+r+1}\T^{n-r-1}$. 

Now we need to vary the argument above to find the last term of order $n+5$, at which time the process terminates. We distinguish terms $\frac{1}{u^pv^q}\Ss^s\T^t$ and $\frac{1}{u^pv^q}\frac{\alpha}{\Delta}\Ss^s\T^t$ by calling them ``function" and ``form" respectively. 
By our process above,
\begin{equation}\label{4}
\nabla_0\left(\frac{1}{u^jv^k}\Ss^{2n}+x_{2n-1,1}\Ss^{2n-1}\T+\cdots+x_{n+7,n-7}\Ss^{n+7}\T^{n-7}\right)
\end{equation}
has bad forms of order at most $n+7$ and at least $n+6$. In fact, any form of order $n+7$ is a linear combination of $\frac{\alpha}{u^3\Delta}\Ss^{n+7}\T^{n-7}$ and $\frac{\alpha}{v^2\Delta}\Ss^{n+7}\T^{n-7}$, neither of which is a bad form. We are thus left with only bad forms of order $n+6$ in (\ref{4}). The only bad form of order $n+6$, up to a rational multiple, is $\frac{1}{uv}\frac{\alpha}{\Delta}\Ss^{n+6}\T^{n-6}$. To cancel these bad forms of order $n+6$, we can choose a rational multiple of $\frac{1}{uv}\Ss^{n+5}\T^{n-5}$ to be the last term of $l_{j,k}$. Therefore, the only possible bad terms in $\nabla_0(l_{j,k})$ comes from the order $n+5$ and $n+4$ parts of $\nabla_0(\frac{1}{uv}\Ss^{n+5}\T^{n-5})$. Both these parts have no bad terms, which can be easily checked by formulas (\ref{2}) and (\ref{3}).

Eventually, we have a linear combination of bad terms $$l_{j,k}:=\frac{1}{u^jv^k}\Ss^{2n}+x_{2n-1,1}\Ss^{2n-1}\T+\cdots+x_{n+7,n-7}\Ss^{n+7}\T^{n-7}+x_{n+5,n-5}\Ss^{n+5}\T^{n-5}$$ such that
$\nabla_0(l_{j,k})$ has no bad terms. After finding this $l_{j,k}$, it is routine to find $\omega_{j,k}^{(0)}$ and $\omega_{j,k}^{(1)}$ by putting all terms with powers of $u$ in the denominator into $\omega_{j,k}^{(0)}$ and putting all terms with powers of $v$ in the denominator into $\omega_{j,k}^{(1)}$. It remains to show that both $\omega_{j,k}^{(0)}$ and $\omega_{j,k}^{(1)}$ are $\nabla_0$-closed. This follows from \cite[Cor. 3.5]{brown-hain}, or one can check directly. 
\end{proof}

\begin{remark}
For different pairs of integers $(j,k)$, the classes $[\wt\omega_{j,k}]\in H^1_\DR(\M_{1,1}, S^{2n}\cH)$ are linearly independent because of Lemma \ref{span} and the fact that $l_{j,k}$ starts with $\frac{1}{u^jv^k}\Ss^{2n}$.
\end{remark}

\begin{example}{\bf The cocycle $\wt\omega_{2,1}$ that represents $[\wt\omega_{2,1}]\in H^1_\DR(\M_{1,1}, S^{14}\cH)$.}\label{msk}
We carry out the process in the above proposition when $(j,k)=(2,1)$ and $n=7$. Starting with $\frac{1}{u^2v}\Ss^{14}$ that represents $[\frac{1}{u^2v}\Ss^{14}]\in H^1_\delta(\U,\F^0_{14})$, we have
$$
\nabla_0\left(\frac{1}{u^2v}\Ss^{14}\right)=-\left(\frac{18}{u^3}\frac{\alpha}{\Delta}+\frac{1}{2v^2}\frac{\alpha}{\Delta}\right)\Ss^{14}-\frac{7}{4uv}\frac{\alpha}{\Delta}\Ss^{13}\T
$$
with a bad term of order $13=n+6$, which can be eliminated by adding a multiple of the function $\frac{1}{uv}\Ss^{12}\T^2$ of order $n+5=12$. Since
$$
\nabla_0\left(\frac{1}{uv}\Ss^{12}\T^2\right)=\frac{3}{uv}\frac{\alpha}{\Delta}-\left(\frac{9}{u^2}\frac{\alpha}{\Delta}+\frac{u}{2v^2}\frac{\alpha}{\Delta}\right)\Ss^{12}\T^2-\frac{3}{2v}\frac{\alpha}{\Delta}\Ss^{11}\T^3,
$$
we choose the correct multiple $\frac{7}{12}$ for $\frac{1}{uv}\T^2\Ss^{12}$, and 
\begin{align*}
\nabla_0\left(\frac{1}{u^2v}\Ss^{14}+\frac{7}{12}\frac{1}{uv}\Ss^{12}\T^2\right)&=-\left(\frac{18}{u^3}\frac{\alpha}{\Delta}+\frac{1}{2v^2}\frac{\alpha}{\Delta}\right)\Ss^{14}-\left(\frac{21}{4u^2}\frac{\alpha}{\Delta}+\frac{7u}{24v^2}\frac{\alpha}{\Delta}\right)\Ss^{12}\T^2
\\&\quad-\frac{7}{8v}\frac{\alpha}{\Delta}\Ss^{11}\T^3
\end{align*}
has no bad terms.

Define $l_{2,1}=\frac{1}{u^2v}\Ss^{14}+\frac{7}{12}\frac{1}{uv}\Ss^{12}\T^2$,
$\omega_{2,1}^{(0)}=\frac{18}{u^3}\frac{\alpha}{\Delta}\Ss^{14}+\frac{21}{4u^2}\frac{\alpha}{\Delta}\Ss^{12}\T^2,$ and $\omega_{2,1}^{(1)}=-\frac{1}{2v^2}\frac{\alpha}{\Delta}\Ss^{14}-\frac{7u}{24v^2}\frac{\alpha}{\Delta}\Ss^{12}\T^2-\frac{7}{8v}\frac{\alpha}{\Delta}\Ss^{11}\T^3$. Then the cocycle
$$\wt{\omega}_{2,1}:=\begin{array}{|cc}
(\omega_{2,1}^{(0)},\omega_{2,1}^{(1)}) &0\\
0& l_{2,1}\\
\hline
\end{array}
$$
is a modular form of the second kind, and represents the class $[\wt\omega_{2,1}]$ we are looking for.
\end{example}

To summarize, we have
\begin{theorem}\label{ADRstr}
For each positive integer $n$, there is a canonical comparison isomorphism
$$\comp_{\B,\DR}: H^1_\DR(\M_{1,1}, S^{2n}\cH)\otimes_\Q\C\xrightarrow{\sim} H^1(\M_{1,1}^\an,S^{2n}\H_\B)\otimes_\Q\C.$$
Moreover, there is a canonical basis for $H^1_\DR(\M_{1,1}, S^{2n}\cH)$, consisting of $[\wt\omega_f]$ for each Hecke eigenform $f$ of weight $2n+2$ and $[\wt\omega_{j,k}]$ for each pair of positive integers $(j,k)$ such that $4j+6k=2n$.
\end{theorem}




\section{$\Q$-structures for the Relative Completion of $\SL_2(\Z)$}\label{qstr}

Suppose that $\Gamma$ is a discrete group, usually the fundamental group of a manifold. Let $R$ be a reductive group defined over $\Q$, and $\rho:\Gamma\to R(\Q)$ a Zariski dense homomorphism. The relative completion of $\Gamma$ with respect to $\rho$ is a proalgebraic group $\cG$, defined over $\Q$, which is an extension 
$$1\to\cU\to\cG\to R\to 1$$  
of $R$ by a prounipotent group $\cU$, together with a Zariski dense homomorphism $\Gamma\to\cG(\Q)$. When $R$ is the trivial group, then $\cU=\cG$ is the unipotent completion of $\Gamma$.

Hain \cite{rc} developed a de Rham theory for the relative completion, which generalizes Chen's $\pi_1$-de Rham theorem from the unipotent case to the relative case. In both cases, the de Rham construction for the fundamental group depends on base points. This dependence is abstract and indirect, but can become evident when we put more structures on the fundamental group, e.g. mixed Hodge structures \cite{rc}. To understand this dependence, we need a more concrete approach. 
 
In this section, to concretely construct the relative completion, we start by giving a useful cohomological criterion, following Hain \cite[\S 4.2]{hdr}. The dependence on base points for the relative completion of a fundamental group naturally leads us to find a canonical universal flat connection (cf. \cite[\S 14.2]{rc}, \cite{cee,kzb,mykzb}). The connection can be found by using a suitable modification of Chen's method of power series connections \cite{chen}. Hain used this method (see Prop. \ref{procedure}) to construct a Betti $\Q$-structure for the relative completion of $\SL_2(\Z)$, which underlies a canonical mixed Hodge structure for each choice of a base point \cite[\S 7]{hdr}. After reviewing his construction in Section \ref{Bstr}, we use essentially the same method to explicitly construct a $\Q$-de Rham structure for the relative completion of $\SL_2(\Z)$. 

We will work complex analytically until Sections \ref{Bstr}--\ref{inter} where we discuss $\Q$-structures.

\subsection{A characterization of relative completion}
This subsection reviews Hain's cohomological criterion, details and proofs can be found in \cite[\S 4.2]{hdr}.

Suppose that $X$ is the orbifold quotient $\Gamma\dbs M$ of a simply connected manifold $M$ by a discrete group $\Gamma$. We view $\Gamma$ as the fundamental group of $X$. Our main example is when $M=\h$ is the upper half plane, $\Gamma=\SL_2(\Z)$, and $X=\M^\an_{1,1}$.

Denote by $\cG$ the relative completion of $\Gamma$ with respect to $\rho:\Gamma\to R$. Denote its unipotent radical by $\cU$. By Levi's splitting theorem, the relative completion can be expressed (non-canonically) as a semi-direct product $\cG\cong R\ltimes\cU$.\footnote{We will always assume $R$ acts on $\cU$ on the left.} We will give a criterion for when a homomorphism $\Gamma\to R\ltimes\cU$ induces an isomorphism $\cG\to R\ltimes\cU$. 

Denote by $\u$ the Lie algebra of $\cU$. Suppose that $\nabla_0$ is a connection on a right principal $\cU$-bundle $M\times\cU\to M$. A 1-form $\omega\in E^1(M)\,\hat\otimes\,\u$ defines a connection
$$\nabla s=\nabla_0 s+\omega s$$
on this bundle, where $s:M\to \cU$ is a section. This connection is flat, i.e. the 1-form $\omega$ being integrable, if and only if
$$\nabla_0\omega+\frac12[\omega,\omega]=0.$$
The 1-form $\omega$ is $\Gamma$-invariant if and only if 
$(\gamma^*\otimes 1)\omega=(1\otimes\gamma)\omega$ for any $\gamma\in\Gamma$. Each integrable $\Gamma$-invariant 1-form $\omega\in E^1(M)\,\hat\otimes\,\u$ defines a homomorphism
\begin{equation}\label{pt}
\tilde{\rho}_x:\Gamma\to R\ltimes\cU
\end{equation}
induced by Chen's parallel transport formula, see details in \cite[Prop. 4.3]{hdr}. 

Suppose $V$ is a finite dimensional $R$-module and that $\cU\to\Aut V$ is an $R$-equivariant homomorphism. Denote by $\V$ the corresponding local system over $X$. Now we are ready to state the criterion. We have the following lemma and proposition \cite[Lem. 4.5, Prop. 4.6]{hdr}.
\begin{lemma}
If $R$ is reductive, each integrable $\Gamma$-invariant 1-form $\omega\in E^1(M)\,\hat\otimes\,\u$ induces a ring homomorphism
$$\theta^*_\omega: [H^\bullet(\u)\otimes V]^R\to H^\bullet(X,\V).$$
\end{lemma}

\begin{proposition}\label{char}
If $R$ is reductive and $\rho:\Gamma\to R$ has Zariski dense image, then the homomorphism $\tilde{\rho}_x:\Gamma\to R\ltimes\cU$ constructed above is the completion of $\Gamma$ with respect to $\rho$ if and only if the homomorphism
$$\theta^*_\omega: [H^j(\u)\otimes V]^R\to H^j(X,\V)$$
is an isomorphism when $j=0,1$ and injective when $j=2$ for any $R$-module $V$.
\end{proposition}

\subsection{The bundle $\bU$ of Lie algebras}\label{la}
One needs to find a connection form that satisfies the conditions in Prop. \ref{char}. We first describe the bundle $\bU$ where the sought-after connection lives.

We continue to use the notation as before, working over $\C$, where $\cG$ is the relative completion of $\Gamma$, along with $\cU$ the unipotent radical, and $\u$ its Lie algebra. In fact, there is an isomorphism
$$H_1(\u)\cong\prod_{\alpha\in\cR} H^1(\Gamma,V_\alpha)^*\otimes V_\alpha,$$
where $\alpha$ runs through the set $\cR$ of isomorphism classes of irreducible $R$-modules. When $\u$ is free, there is a non-canonical isomorphism 
$$\u\cong\LL\big(\bigoplus_\alpha H^1(\Gamma,V_\alpha)^*\otimes V_\alpha\big)^\wedge.$$
This is true for our main example, with $\Gamma=\SL_2(\Z)$, $R=\SL_2$. In this case, the Lie algebra $\u^{\rel}$ of the unipotent radical $\cU^{\rel}$ of the relative completion $\cG^\rel$ is freely topologically generated by 
$$\prod_{n\ge 0} H^1(\text{SL}_2(\mathbb Z), S^{2n} H)^*\otimes S^{2n} H,$$
where $H$ is the standard representation of $\text{SL}_2$, and $S^{2n} H$ its $2n$-th symmetric power.

Suppose that $X$ is an affine curve whose fundamental group is isomorphic to $\Gamma$. For any base point $x\in X$, we have a Zariski dense monodromy representation
$$\rho_x:\pi_1(X,x)\to R_x(\Q)$$
where $R_x$ is isomorphic to $R$. Denote the relative completion of $\pi_1(X,x)$ with respect to $\rho_x$ by $\cG_x$. Denote the unipotent radical by $\cU_x$, and its Lie algebra by $\u_x$. Then $\cG_x\cong\cG$, $\cU_x\cong\cU$, and $\u_x\cong\u$. For each $R$-module $V_\alpha$, denote by $\V_\alpha$ the corresponding local system over $X$. We can identify $H^1(\Gamma,V_\alpha)$ with $H^1(X, \V_\alpha)$.

Define
$$\bu_1:=\prod_{\alpha\in \cR} H^1(X, \V_\alpha)^*\otimes \V_\alpha.$$
This is a (pro-)local system over $X$, whose fiber over $x$ is the abelianization of the free Lie algebra $\u_x$.

The degree $n$ part $V\mapsto\LL_n(V)$ of the free Lie algebra is a Schur functor, so that it makes sense to apply it to local systems. Define
$$\bu_n=\LL_n(\bu_1):=\varprojlim_n\LL_n\big(\bigoplus_\alpha H^1(X,\V_\alpha)^*\otimes \V_\alpha\big).$$
For any $u\in\LL(\bu_1)$, we will often denote its degree $n$ part by $(u)_n$.
Set $$\bu=\LL(\bu_1):=\varprojlim_n\bigoplus_{j=1}^n \bu_j,$$
and $$\bu^N:=\varprojlim_{n\ge N}\bigoplus_{j=N}^n \bu_j,$$
the parts of degree at least $N$ in the Lie algebra $\bu$. Define 
$$\bU:=\bu\otimes\OO_X$$ the vector bundle over $X$. Similarly, define vector bundles
$$\bU_n:=\bu_n\otimes\OO_X\qquad\text{ and }\qquad\bU^N:=\bu^N\otimes\OO_X$$
over $X$. There is a filtration 
$$\bU=\bU^1\supset\bU^2\supset\bU^3\supset\cdots$$
on the bundle $\bU$. This filtration, when restricted to a fiber, is the lower central series of the fiber. Note that 
$$\Gr^\bullet_\LCS\bU:=\bU^n/\bU^{n+1}$$
is naturally isomorphic to $\bU_n$. For each $N\ge 1$, the bundle
$$\bU/\bU^N\cong\bU_1\oplus\cdots\oplus\bU_N$$
is flat with monodromy that factors through $\rho_x$. Denote the limit of these flat connections by $\nabla_0$. In the next subsection, we will find an integrable 1-form $\Omega$ so that 
$$\nabla:=\nabla_0+\Omega$$
defines a flat connection on the bundle $\bU$ over $X$.

\subsection{The connection form $\Omega$}\label{connform}
This subsection reviews Hain's approach to find a canonical connection form \cite[\S 7.3]{hdr}, which is modified from Chen's method of power series connections \cite{chen}.

Suppose that $K^\bullet(X)$ is a commutative differential graded algebra (cdga) whose cohomology computes $H^\bullet(X)$, for example, the (analytic) de Rham complex $E^\bullet(X)$ or the (algebraic) logarithmic de Rham complex $\Omega^\bullet_\Xbar(\log C)$ for $X=\Xbar-C$. Suppose that $\bu$ is a free Lie algebra generated by $\bu_1$.\footnote{In fact, $\bu$ can also be a local system or a vector bundle of Lie algebras such that $\bu=\LL(\bu_1)$.} Define a differential graded Lie algebra
$$K^\bullet(X;\bu)=K^\bullet(X)\otimes\bu,$$
where the differential $D$ is induced from the differential of the cdga $K^\bullet(X)$.
Define a bracket $[\cdot,\cdot]$ on $K^\bullet(X;\bu)$ by:
$$[\Omega_\beta\otimes u,\Omega_\gamma\otimes v]:=(\Omega_\beta\wedge\Omega_\gamma)\otimes[u,v],$$
which is induced from the wedge product $\wedge$ on $K^\bullet(X)$ and the Lie bracket $[\cdot,\cdot]$ on $\bu$.

The following result is used to inductively construct the connection form $\Omega$. It is a variant of Chen's method of power series connections \cite{chen}. 
\begin{proposition}\label{procedure}
Suppose that $H^2(K^\bullet(X))=0$, and that we have a closed form
$$\Omega_1\in K^1(X;\bu_1).$$
For each $n\geq 2$, we can find $\Xi_n\in K^1(X;\bu_n)$, and set $$\Omega_n:=\Omega_{n-1}+\Xi_n,$$ so that
$$D\Omega_n+\frac 12[\Omega_n,\Omega_n]\equiv 0 \mod \bu^{n+1}.$$
\end{proposition}
\begin{proof}
Since $D\Omega_1=0$, we have that $[\Omega_1,\Omega_1]\in K^2(X;\bu_2)$ is closed. Since $H^2(K^\bullet(X))=0$, the form $[\Omega_1,\Omega_1]$ is thus exact. One can find $\Xi_2\in K^1(X;\bu_2)$ such that $-D\Xi_2=\frac 12[\Omega_1,\Omega_1]$ and $D\Omega_2+\frac 12[\Omega_2,\Omega_2]\equiv 0 \mod \bu^{3}$.

Suppose for $n\geq 2$ we have already found $\Xi_2,\cdots,\Xi_{n}$, and $\Omega_{n}=\Omega_1+\sum_{i=2}^{n}\Xi_i$, such that $D\Omega_{n}+\frac 12[\Omega_{n},\Omega_{n}]\equiv 0 \mod\bu^{n+1}.$ We claim that the degree $(n+1)$ part 
$$(D\Omega_{n}+\frac 12[\Omega_{n},\Omega_{n}])_{n+1}\in K^2(X;\bu_{n+1})$$ is closed. In fact,
\begin{align*}
D(D\Omega_{n}+\frac 12[\Omega_{n},\Omega_{n}])&= \frac 12[D\Omega_{n},\Omega_{n}]-\frac 12[\Omega_{n},D\Omega_{n}]
\\&\equiv\frac 12[(-\frac 12[\Omega_{n},\Omega_{n}]),\Omega_{n}]-\frac 12[\Omega_{n},(-\frac 12[\Omega_{n},\Omega_{n}])]
\\&=-\frac 14[[\Omega_{n},\Omega_{n}],\Omega_{n}]+\frac 14[\Omega_{n},[\Omega_{n},\Omega_{n}]]=0 \mod \bu^{n+2}
\end{align*}
where we have used Leibniz rule of $D$ on the first line, induction hypothesis on the second line, and both terms on the last line are 0 by Jacobi identity. Since $H^2(K^\bullet(X))=0$, the form $(D\Omega_{n}+\frac 12[\Omega_{n},\Omega_{n}])_{n+1}\in K^2(X;\bu_{n+1})$ is thus exact. We can find $\Xi_{n+1}\in K^1(X;\bu_{n+1})$ such that 
$$-D\Xi_{n+1}=(D\Omega_{n}+\frac 12[\Omega_{n},\Omega_{n}])_{n+1}=(\frac 12[\Omega_{n},\Omega_{n}])_{n+1}.$$ 
Define $\Omega_{n+1}:=\Omega_n+\Xi_{n+1}$, then it is easy to check that 
$$D\Omega_{n+1}+\frac 12[\Omega_{n+1},\Omega_{n+1}]\equiv 0 \mod\bu^{n+2}$$
since $D\Omega_{n+1}+\frac 12[\Omega_{n+1},\Omega_{n+1}]\equiv D\Omega_{n}+\frac 12[\Omega_{n},\Omega_{n}]\equiv 0 \mod\bu^{n+1}$ and the degree $(n+1)$ part on the left is $D\Xi_{n+1}+(\frac 12[\Omega_{n},\Omega_{n}])_{n+1}=0$.
\end{proof}

By the proposition, one can define the 1-form
$$\Omega:=\varprojlim_n\Omega_n\in K^1(X;\bu).$$
Then it is integrable, i.e. it satisfies that $$D\Omega+\frac 12[\Omega,\Omega]=0.$$

\subsection{Betti $\Q$-structure for the relative completion}\label{Bstr}
From now on we shall work over $\Q$. We construct the Betti $\Q$-structure for the relative completion, following Hain \cite[\S 7.6]{hdr}. For any base point $x\in X$, we have a Zariski dense monodromy representation $\rho_x:\pi_1(X,x)\to R_x(\Q)$. Denote the relative completion over $\Q$ by $\cG_x$. Denote the Lie algebra of its unipotent radical $\cU_x$ by $\u_x$. These Lie algebras form a local system $\bu_\B$ over $X$. Their abelianizations also form a local system $\bu_{1,\B}$ over $X$. In our main example, we have
$$\bu_{1,\B}=\prod_{n\ge 0} H^1(\M^\an_{1,1},S^{2n}\H_\B)^*\otimes S^{2n}\H_\B.$$

Since $H^1(X,H^1(X,\V_\alpha)^*\otimes\V_\alpha)$ is naturally isomorphic to 
$$H^1(X,\V_\alpha)^*\otimes H^1(X,\V_\alpha)\cong\Hom(H^1(X,\V_\alpha),H^1(X,\V_\alpha)).$$
For each $\alpha\in\cR$, the identity map $H^1(X,\V_\alpha)\to H^1(X,\V_\alpha)$ can be represented by a closed 1-form
$$\omega_\alpha\in K^1(X,H^1(X,\V_\alpha)^*\otimes\V_\alpha),$$
where $K^\bullet(X)$ is the analytic de Rham complex $E^\bullet(X)$ of $X$. Set 
$$\Omega_1:=\prod_{\alpha\in\cR}\omega_\alpha\in K^1(X;\bu_1).$$

Suppose that $X$ is an affine curve, then $H^2(K^\bullet(X))$ vanishes. Since $\Omega_1$ is closed, by Proposition \ref{procedure}, one can find a canonical 1-form $\Omega$ which is integrable. Then $$\nabla_\B:=\nabla_0+\Omega$$
defines a flat \emph{Betti} connection on the holomorphic vector bundle $\bU\cong\bu_\B\otimes_\Q\OO_{X^\an}$ whose fiber over $x$ is identified with the complexification of the Lie algebra $\u_x$, and we view $\Omega$ as acting on each fiber by inner derivations. The parallel transport of $\Omega$ gives rise to a homomorphism (cf. (\ref{pt}))
\begin{equation}\label{rhox}
\tilde{\rho}_x:\pi_1(X,x)\to R\ltimes\cU\cong\cG_\C,
\end{equation}
where $\cG_\C$ denotes the relative completion of $\Gamma$ over $\C$. Base change \cite[\S 3.3]{hdr} implies there is a natural isomorphism $\cG\times_\Q \C\to\cG_\C$. By Proposition \ref{char} and the fact that $\Omega_1$ represents identity maps $H^1(X,\V_\alpha)\to H^1(X,\V_\alpha)$ for all $\alpha\in\cR$, we have that $\tilde{\rho}_x$ induces an isomorphism 
\begin{equation}\label{thetax}
\Theta_x:\cG^\B_x\times_\Q \C\to \cG_\C.
\end{equation}
This provides a Betti $\Q$-structure for the relative completion. In particular, by strictness, Hain \cite{hdr} carefully chooses the 1-forms $\Omega_1=\prod_\alpha \omega_\alpha$ and subsequently $\Omega$ that are compatible with the weight and Hodge filtrations, so that the constructed Betti $\Q$-structure underlies a canonical mixed Hodge structure.

\begin{remark}[Orbifold case]
Suppose that $X=\Xbar-C$ is an orbi-curve, i.e. the orbifold quotient of an affine curve $Y=\Ybar-D$ by a group $G$. Then $G$ acts on the de Rham complex $\Omega^\bullet_Y$ and the logarithmic de Rham complex $\Omega^\bullet_\Ybar(\log D)$ of $Y$. To work with $X$ is to work $G$-equivariantly with $Y$. One chooses either $(\Omega^\bullet_Y)^G$ or $(\Omega^\bullet_\Ybar(\log D))^G$ as $K^\bullet(X)$.
\end{remark}

Note that the orbifold case is applicable to a modular curve. Therefore, we get a Betti $\Q$-structure for the relative completion of its fundamental group, i.e. a modular group \cite{hdr}.

\subsection{$\Q$-de Rham structure for the relative completion}\label{coc}
For the $\Q$-de Rham structure, we focus on our main example: the relative completion of $\SL_2(\Z)$. We continue to work over $\Q$. Recall that $X=\M_{1,1}$ and $\Xbar=\Mbar_{1,1}$ are the stack quotient of $Y=\A^2-D$ and $\Ybar=\A^2-\{(0,0)\}$ by a $\G_m$-action. Set $K^\bullet(\M_{1,1})$ as the single complex associated to the $\rmC$ech--de Rham complex $\cC^\bullet(\U,(\Omega^\bullet_\Ybar(\log D))^{\G_m})$.\footnote{This is the case $n=0$ for the $\rmC$ech--de Rham complex $\cC^\bullet(\U,\F^\bullet_{2n})$, see definition in Section \ref{cech}.} We find in this subsection an explicit $\rmC$ech--de Rham 1-cochain $\wt\Omega$, which serves the same purpose as $\Omega$ in the last subsection for the Betti $\Q$-structure. We will show how to interpretate $\wt\Omega$ as a connection form over $\M_{1,1}$ next in Section \ref{inter}. From this, one constructs a $\Q$-de Rham structure for the relative completion of $\SL_2(\Z)$.

Define an algebraic (pro-)vector bundle over $\M_{1,1}$
\begin{equation}\label{u1dr}
\bu_{1,\DR}:=\prod_{n\ge 0} H^1_\DR(\M_{1,1}, S^{2n}\cH)^*\otimes S^{2n}\cH.
\end{equation}
One can then define bundles $\bu_{n,\DR}$, $\bu^N_\DR$, and $\bu_\DR$ (cf. the definitions for the local systems $\bu_n$, $\bu^N$, and $\bu$ in Section \ref{la}). There is a filtration
$$\bu_\DR=\bu^1_\DR\supset\bu^2_\DR\supset\bu^3_\DR\supset\cdots$$
on the bundle $\bu_\DR$ induced by lower central series on each fiber. For each $n\ge 1$, the graded quotient $\bu^n_\DR/\bu^{n+1}_\DR$ is naturally isomorphic to $\bu_{n,\DR}$, and is equipped with a flat connection $\nabla_0$ induced from the Gauss--Manin connection on $\cH$ defined in Section \ref{gaussmanin}. Summing these, we have that for each $N\ge 1$ the bundle
$$\bu_\DR/\bu^N_\DR\cong\bu_{1,\DR}\oplus\cdots\bu_{N,\DR}$$
is equipped with a flat connection. Denote the limit of these flat connections on $\bu_\DR$ by $\nabla_0$. 

Recall from (\ref{reladr}) that there is a natural isomorphism
$$
\cH\otimes_\Q \C\cong \H_\B\otimes_\Q \OO_{Y^\an},
$$
of bundles with connection. This induces an isomorphism of bundles with connection over $\M^\an_{1,1}$
$$\bu_{1,\DR}\otimes_\Q\C\cong\bu_{1,\B}\otimes_\Q\OO_{\M_{1,1}^\an}.$$
The connection on each side is essentially the Gauss--Manin connection, and is denoted by $\nabla_0$.
If we apply the functor $V\mapsto\LL(V)$, we have an isomorphism of the underlying holomorphic vector bundles
$$\bu_\DR\otimes_\Q\C\cong\bu_\B\otimes_\Q\OO_{\M_{1,1}^\an}.$$
However, this is not an isomorphism of bundles with connections. The bundle on the right hand side $\bU:=\bu_\B\otimes_\Q\OO_{\M_{1,1}^\an}$ is equipped with $\nabla_\B=\nabla_0+\Omega$, where $\Omega$ acts on each fiber by inner derivations. The bundle on the left hand side $\bu_\DR$ with connection $\nabla_0$ is isomorphic to the associated graded $(\Gr^\bullet_\LCS\bU,\nabla_0)$, but not to $(\bU,\nabla_\B)$.\footnote{Since $\Omega$ acts trivially on the graded pieces of $\bU$, we have that the induced connection of $\nabla_\B$ on the associated graded $\Gr^\bullet_\LCS\bU$ is $\nabla_0$, \cite[Lem. 7.4]{hdr}.} 
We will construct a connection form $\wt\Omega$ similar to $\Omega$ on the bundle $\bu_\DR$, which acts on each fiber by inner derivations.

We start by defining a 1-cocycle in $K^1(\M_{1,1};\bu_{1,\DR})$ 
$$\wt\Omega_1=\begin{array}{|cc}
(\Omega_1^{(0)},\Omega_1^{(1)}) &0\\
0& f_1\\
\hline
\end{array}
$$
that represents the identity maps $H^1_\DR(\M_{1,1},S^{2n}\cH)\to H^1_\DR(\M_{1,1},S^{2n}\cH)$ for every $n>0$. It can be written as
$$\wt\Omega_1:=\prod_{n\ge 1}\wt\Omega_{1,2n}\in K^1(\M_{1,1};\bu_{1,\DR}),$$
where we define $\wt\Omega_{1,2n}$ by using 1-cocyles $\wt\omega_f$, $\wt\omega_{j,k}$ found in Section \ref{cmfano},
$$
\wt\Omega_{1,2n}:=\sum_{f}\wt\omega_f\X_f +\sum_{\{(j,k): 4j+6k=2n\}}\wt\omega_{j,k}\X_{j,k}\in K^1(\M_{1,1}; S^{2n}\cH)\otimes H^1_\DR(\M_{1,1}, S^{2n}\cH)^*,
$$
with the first term on the right hand side is summing over Hecke eigenforms $f$ of weight $2n+2$, and all $\X_f,\X_{j,k}\in H^1_\DR(\M_{1,1}, S^{2n}\cH)^*$ form a dual basis for all $\Q$-de Rham classes $[\wt\omega_f]$ and $[\wt\omega_{j,k}]$ in $H^1_\DR(\M_{1,1}, S^{2n}\cH)$. 

For our purpose, we would prefer to write $\wt\Omega_{1,2n}$ in a different way: let
\begin{equation}\label{dw}
\wt\omega_f=\wt\Omega_f\T^{2n}\quad\text{and}\quad\wt\omega_{j,k}=\sum_{m}\wt\Omega^m_{j,k}\Ss^m\T^{2n-m},
\end{equation}
with coefficients $\wt\Omega_f$, and $\wt\Omega^m_{j,k}\in K^1(\M_{1,1}), 0\leq m\leq 2n$. Set
$$\e_f:=\X_f\otimes\T^{2n}\quad\text{and}\quad\e^m_{j,k}:=\X_{j,k}\otimes\Ss^m\T^{2n-m},$$
then $\e_f,\e^m_{j,k}\in\bu_{1,\DR}$, and we can rewrite
$$
\wt\Omega_{1,2n}=\sum_{f}\wt\Omega_f\e_f +\sum_{\{(j,k): 4j+6k=2n\}}\left(\sum_m\wt\Omega^m_{j,k}\e^m_{j,k}\right)\in K^1(\M_{1,1};\bu_{1,\DR}).
$$

Note that $H^2(K^\bullet(\M_{1,1}))$ vanishes.\footnote{By equation (3.3) in \cite{brown-hain} and the fact that $H^j(\M_{1,1},\Q)$ vanishes for $j=1,2$.} Using the procedure in Proposition \ref{procedure}, from the 1-cocycle $\wt\Omega_1\in K^1(\M_{1,1};\bu_{1,\DR})$, we can find a 1-cochain 
$$\wt\Xi_n=\begin{array}{|cc}
(\Xi_n^{(0)},\Xi_n^{(1)}) &0\\
0& f_i\\
\hline
\end{array}
$$
in $K^1(\M_{1,1};\bu_{n,\DR})$, and set  
$$\wt\Omega_n=\begin{array}{|cc}
(\Omega_n^{(0)},\Omega_n^{(1)}) &0\\
0& F_n\\
\hline
\end{array}
$$
recursively as
$$\wt\Omega_n:=\wt\Omega_{n-1}+\wt\Xi_n$$
for each $n\ge 2$, so that
\begin{equation}\label{recur}
-D\wt\Xi_n=(\frac 12[\wt\Omega_{n-1},\wt\Omega_{n-1}])_n.
\end{equation}
 Taking the limit, we have a 1-cochain
$$\wt\Omega=\begin{array}{|cc}
(\Omega^{(0)},\Omega^{(1)}) &0\\
0& F\\
\hline
\end{array}
:=\varprojlim\limits_n\wt\Omega_n=\begin{array}{|cc}
(\varprojlim\limits_n\Omega_n^{(0)},\varprojlim\limits_n\Omega_n^{(1)}) &0\\
0& \varprojlim\limits_n F_n\\
\hline
\end{array}
$$
 in $K^1(\M_{1,1};\bu_\DR)$ such that 
$$D\wt\Omega+\frac 12[\wt\Omega,\wt\Omega]=0.$$

\subsection{Interpretation for $\wt\Omega$ as a connection form}\label{inter}

In this subsection, we interpretate $\wt\Omega\in K^1(\M_{1,1};\bu_\DR)$ as a connection form on the bundle $\bu_\DR$. This naturally provides a $\Q$-de Rham structure for the relative completion of $\SL_2(\Z)$. Note that
$$\wt\Omega=\begin{array}{|cc}
(\Omega^{(0)},\Omega^{(1)}) &0\\
0& F\\
\hline
\end{array}
$$
has three components where $\Omega^{(j)}\in\F^1_0(U_j)\otimes\bu_\DR$, $j=0,1$ and $F\in\F^0_0(U_{01})\otimes\bu_\DR$. Recall that each fiber of the bundle $\bu_\DR$ is isomorphic to $\u^\rel$, and one can view $\bu_\DR$ as a principal $\cU^\rel$-bundle.\footnote{A prounipotent group is naturally isomorphic to its Lie algebra as a group.} We show that there is a gauge transformation $$g:U_{01}\to\cU^\rel\xhookrightarrow{}\Aut(\u^\rel)$$
between the flat connections $$\nabla^{(j)}=\nabla_0+\Omega^{(j)}$$ on the intersection $U_{01}$ of the open subsets $U_j$ of $\Mbar_{1,1}$ given by $g:=1+F\in\cU^\rel(\OO(U_{01}))$.\footnote{Since $\cU^\rel$ is an proalgebraic group, $g$ can be viewed as a $\OO(U_{01})$-point.} By gluing the connections $\nabla^{(j)}$ together on $U_{01}$ via this gauge transformation, we get a flat $\Q$-de Rham connection $$\nabla_\DR=\nabla_0+\wt\Omega$$ on the bundle $\bu_\DR$. Moreover, since we have used the logarithm de Rham complex $(\Omega^\bullet_\Ybar(\log D))^{\G_m}$ for $K^\bullet(\M_{1,1})$, the bundle $\bu_\DR$ we have constructed extends to $\overline\bu_\DR$ over $\Mbar_{1,1}$, which is endowed with the same flat connection $\nabla_\DR$ that is defined over $\Q$ and has regular singularity at the cusp.

To start with, we define a wedge product on $K^\bullet(\M_{1,1})$, following the convention in \cite[(14.24)]{bott-tu}. For example, if $\wt\Omega$ and $\wt\Omega'$ are 1-cochains in $K^1(\M_{1,1})$ with 
$$\wt{\Omega}=\begin{array}{|cc}
(\omega^{(0)},\omega^{(1)}) &0\\
0& l\\
\hline
\end{array}
\qquad
\text{and}
\qquad
\wt\Omega'=\begin{array}{|cc}
(\omega'^{(0)},\omega'^{(1)}) &0\\
0& l'\\
\hline
\end{array}
$$
then their product is given by
\begin{equation}\label{wp}
\wt{\Omega}\wedge\wt\Omega':=\begin{array}{|cc}
(\omega^{(0)}\wedge\omega'^{(0)},\omega^{(1)}\wedge\omega'^{(1)}) &0\\
0& l\cdot\omega'^{(1)}-\omega^{(0)}\cdot l'\\
0& 0\\
\hline
\end{array}
\end{equation}

Since $D\wt\Omega+\frac 12[\wt\Omega,\wt\Omega]\in K^2(\M_{1,1};\bu_\DR)$ a priori has two parts
$$\begin{array}{|cc}
* & 0 \\
0 & * \\
0 & 0 \\
\hline
\end{array}
$$
The top part being $0$ means $\nabla_0\Omega^{(0)}+\Omega^{(0)}\wedge\Omega^{(0)}=0$ and $\nabla_0\Omega^{(1)}+\Omega^{(1)}\wedge\Omega^{(1)}=0$, which tells us that $\Omega^{(j)}$ is integrable on $U_j$ for $j=0,1$. They give rise to flat connections $\nabla^{(j)}=\nabla_0+\Omega^{(j)}$ on $U_j$. The lower part being $0$ means
\begin{equation}\label{gauge}
\Omega^{(1)}-\Omega^{(0)}-\nabla_0 F+F\cdot\Omega^{(1)}-\Omega^{(0)}\cdot F=0,
\end{equation}
where ``$\cdot$" is a product on $K^\bullet(\M_{1,1})\otimes\bu_\DR$ induced by the wedge product on $K^\bullet(\M_{1,1})$ and the Lie bracket on $\bu_\DR$. In fact, this equation (\ref{gauge}) tells us that the connection forms $\Omega^{(0)}$ and $\Omega^{(1)}$ are gauge equivalent on the intersection $U_{01}$ of $U_0$ and $U_1$. 
\begin{proposition}
The function $$g:=1+F\in\cU^\rel(\OO(U_{01}))$$ gauge transforms $\Omega^{(1)}$ to $\Omega^{(0)}$. That is, on $U_{01}$ we have $$\Omega^{(0)}=-\nabla_0 g\cdot g^{-1}+g\,\Omega^{(1)}\,g^{-1}.$$
\end{proposition}
\begin{proof}
Let $g_n:=1+F_n$, then it suffices to prove for every $n$, we have
\begin{equation}\label{gn}
\Omega_{n}^{(0)} \equiv -\nabla_0 g_{n}\cdot g_{n}^{-1}+g_{n}\,\Omega_{n}^{(1)}\,g_{n}^{-1}\qquad\mod \bu^{n+1}_\DR.
\end{equation}
We prove this by induction. When $n=1$, as $\nabla_0(1)=d(1)=0$ and $g_1^{-1}\equiv 1\mod \bu^1_\DR$, (\ref{gn}) becomes
$$\Omega_1^{(0)}=-\nabla_0 f_1+\Omega_1^{(1)},$$
which amounts to the fact that
$$\wt\Omega_1=\begin{array}{|cc}
(\Omega_1^{(0)},\Omega_1^{(1)}) &0\\
0& f_1\\
\hline
\end{array}
$$
is $D$-closed.

Assume that (\ref{gn}) holds for $(n-1)$, we have
\begin{equation}\label{indhyp}
\Omega_{n-1}^{(0)} \equiv -\nabla_0 g_{n-1}\cdot g_{n-1}^{-1}+g_{n-1}\,\Omega_{n-1}^{(1)}\,g_{n-1}^{-1}\qquad\mod \bu^{n}_\DR.
\end{equation}
As for $n$, we only need to prove that the degree $n$ parts on both sides of (\ref{gn}) are the same. On the left hand side, it is $\Xi_n^{(0)}$. One can easily show that $g_n^{-1}\equiv g_{n-1}^{-1}\mod\bu^n_\DR$, so we write the right hand side as
\begin{multline*}
-\nabla_0 g_{n}\cdot g_{n}^{-1}+g_{n}\,\Omega_{n}^{(1)}\,g_{n}^{-1}\\= -\nabla_0 (g_{n-1}+f_n)\cdot (g_{n-1}^{-1}+u_n)+(g_{n-1}+f_n)\,(\Omega_{n-1}^{(1)}+\Xi_n^{(1)})\,(g_{n-1}^{-1}+u_n)
\end{multline*}
with some $u_n\in\bu^{n}_\DR$. Modulo terms in $\bu^{n+1}_\DR$, the degree $n$ part on the right side comes from
$$
-\nabla_0 g_{n-1}\cdot g_{n-1}^{-1}-\nabla_0 f_n+g_{n-1}\,\Omega_{n-1}^{(1)}\,g_{n-1}^{-1}+\Xi_n^{(1)}.
$$
Or equivalently, the degree $n$ part on the right hand side is
$$
\Xi_n^{(1)}-\nabla_0 f_n+(-\nabla_0 g_{n-1}\cdot g_{n-1}^{-1}+g_{n-1}\,\Omega_{n-1}^{(1)}\,g_{n-1}^{-1})_n.
$$

It remains to prove that the above is the same as $\Xi_n^{(0)}$, which is equivalent to showing that
\begin{equation}\label{indstep}
\Xi_n^{(1)}-\Xi_n^{(0)}-\nabla_0 f_n=-(-\nabla_0 g_{n-1}\cdot g_{n-1}^{-1}+g_{n-1}\,\Omega_{n-1}^{(1)}\,g_{n-1}^{-1})_n.
\end{equation}
Note that $\Omega_{n-1}^{(0)}$ has terms only of degree less than $n$, so one can add it to the right hand side without affecting the equality:
$$(-\nabla_0 g_{n-1}\cdot g_{n-1}^{-1}+g_{n-1}\,\Omega_{n-1}^{(1)}\,g_{n-1}^{-1})_n=(-\nabla_0 g_{n-1}\cdot g_{n-1}^{-1}+g_{n-1}\,\Omega_{n-1}^{(1)}\,g_{n-1}^{-1}-\Omega_{n-1}^{(0)})_n.
$$
By the induction hypothesis (\ref{indhyp}), the form in the parenthesis on the right has terms of degree $n$ or higher. When multiplied by $g_{n-1}$ on the right, its degree $n$ part remains unchanged:
$$
(-\nabla_0 g_{n-1}\cdot g_{n-1}^{-1}+g_{n-1}\,\Omega_{n-1}^{(1)}\,g_{n-1}^{-1}-\Omega_{n-1}^{(0)})_n=(-\nabla_0 g_{n-1}+g_{n-1}\cdot\Omega_{n-1}^{(1)}-\Omega_{n-1}^{(0)}\cdot g_{n-1})_n.
$$
Since $g_{n-1}$, and thus $\nabla_0 g_{n-1}$, has terms only of degree less than $n$, we can remove them:
$$
(-\nabla_0 g_{n-1}+g_{n-1}\cdot\Omega_{n-1}^{(1)}-\Omega_{n-1}^{(0)}\cdot g_{n-1})_n=(g_{n-1}\cdot\Omega_{n-1}^{(1)}-\Omega_{n-1}^{(0)}\cdot g_{n-1})_n.
$$
Plugging in $g_{n-1}=1+F_{n-1}$, and then removing the terms $\Omega_{n-1}^{(1)}-\Omega_{n-1}^{(0)}$ of degree less than $n$, we get
$$
(g_{n-1}\cdot\Omega_{n-1}^{(1)}-\Omega_{n-1}^{(0)}\cdot g_{n-1})_n=(F_{n-1}\cdot\Omega_{n-1}^{(1)}-\Omega_{n-1}^{(0)}\cdot F_{n-1})_n.
$$
The equation (\ref{indstep}) now reduces to the equation
$$
\Xi_n^{(1)}-\Xi_n^{(0)}-\nabla_0 f_n+(F_{n-1}\cdot\Omega_{n-1}^{(1)}-\Omega_{n-1}^{(0)}\cdot F_{n-1})_n=0.
$$
This equation holds as it is the equation (\ref{recur}) in the recursive definition of $\wt\Xi_n$ (cf. it also follows from the degree $n$ part of the equation (\ref{gauge})).
\end{proof}

According to the proposition, we can glue the connections $\nabla^{(j)}$ together on $U_{01}$ via this gauge transformation $g$. We obtain a flat $\Q$-de Rham connection $$\nabla_\DR=\nabla_0+\wt\Omega$$ on the bundle $\bu_\DR$ over $\M_{1,1}$.
\begin{theorem}\label{QDRstr}
The de Rham vector bundle $(\bu_\DR,\nabla_\DR)$ can be used to construct a $\Q$-de Rham structure on the relative completion $\cG^\rel$ of $\SL_2(\Z)$ for each base point $x\in\M_{1,1}(\Q)$.
\end{theorem}
\begin{proof}
Note that the 1-forms $\Omega^{(0)}$ and $\Omega^{(1)}$ in $\wt\Omega$ are $\G_m$-invariant on $Y^\an$. After we pull back these 1-forms along the map $\psi$ (defined in (\ref{psi})), we obtain $\SL_2(\Z)$-invariant 1-forms on the upper half plane $\h$. We can now define transport formula on $\h$ by using the pullback of integrable connection forms $\Omega^{(0)}$ and $\Omega^{(1)}$ on opens $U_0$ and $U_1$ respectively, then patching things together on their intersection via the gauge transformation $g$. This gives rise to a homomorphism (cf. homomorphism (\ref{rhox}))
$$
\tilde{\rho}_x:\pi_1(\M^\an_{1,1},x)\to \SL_2\ltimes\cU^\rel\cong\cG^\rel_\C.
$$ 
By Proposition \ref{char} and the fact that $\wt\Omega_1$ represents identity maps on $H^1_\DR(\M_{1,1},S^{2n}\cH)$ for every $n\ge0$, we have that $\tilde{\rho}_x$ induces an isomorphism (cf. isomorphism (\ref{thetax}))
\begin{equation}\label{theta}
\Theta_x:\cG^\DR_x\times_\Q \C\to \cG^\rel_\C.
\end{equation}
This provides a $\Q$-de Rham structure on the relative completion of $\SL_2(\Z)$ for each base point $x\in\M_{1,1}(\Q)$.
\end{proof}
Combined with the isomorphism (\ref{thetax}) in our case, we have the following result. It was proved in \cite[\S 12]{brown:mmm} using Riemann--Hilbert correspondence and tannakian formalism.
\begin{corollary}
There is a comparison isomorphism
$$\comp_{\B,\DR}:\cG^\DR_x\times_\Q \C\to\cG^\B_x\times_\Q \C$$
for each base point $x\in\M_{1,1}(\Q)$.
\end{corollary}
\begin{remark}\label{tv}
In particular, one can choose the base point $x$ to be the unit tangent vector $\partial/\partial q$ at the cusp. The newly constructed $\Q$-de Rham structure and the comparison isomorphism will allow us to compute periods of the relative completion of $\SL_2(\Z)$, i.e. multiple modular values \cite{brown:mmm}, in the next section.
\end{remark}

Finally, we show that the $\Q$-de Rham structure we constructed is canonical. Replacing $S^{n}\cH$ by $S^n\Hbar$ in the definition (\ref{u1dr}) of $\bu_{1,\DR}$, one naturally extends $\bu_\DR$ to a bundle $\overline\bu_\DR$ over $\Mbar_{1,1}$. Recall that there is a Betti vector bundle $\bU=\bu_\B\otimes_\Q\OO_{\M^\an_{1,1}}$ with its flat Betti connection $\nabla_\B=\nabla_0+\Omega$. Denote by $(\overline\bU,\nabla_\B)$ Deligne's canonical extension of $(\bU,\nabla_\B)$ to $\Mbar_{1,1}$.
\begin{theorem}\label{DRbdle}
There is an algebraic de Rham vector bundle $\overline\bu_\DR$ over $\Mbar_{1,1/\Q}$ endowed with connection
$$\nabla_\DR=\nabla_0+\wt\Omega,$$
and an isomorphism 
$$(\overline\bu_\DR,\nabla_\DR)\otimes_\Q\C\approx(\overline\bU,\nabla_\B).$$
The algebraic de Rham vector bundle $\overline\bu_\DR$ and its connection $\nabla_\DR$ are both defined over $\Q$. Moreover, the connection $\nabla_\DR$ has a regular singularity at the cusp.
\end{theorem}
\begin{proof}
The last two statements follow from the construction. For the first statement, note that the extension of a bundle from $\M_{1,1}$ to $\Mbar_{1,1}$ is natural. Therefore, one only needs to prove the isomorphism 
$$(\bu_\DR,\nabla_\DR)\otimes_\Q\C\approx(\bU,\nabla_\B)$$
over $\M_{1,1}^\an$. Now we work over $\C$. Since both bundles are principal $\cU^\rel$-bundles, any fiber is isomorphic to $\cU^\rel$. Fix a base point $x\in\M_{1,1}^\an$. Identify the fiber over $x$ on both bundles with $\cU^\rel$, using (\ref{theta}) and (\ref{thetax}). We need to show that the monodromy representations
$$\rho^\DR_x:\pi_1(\M_{1,1}^\an,x)\to\Aut(\cU^\rel)\quad\text{and}\quad\rho^\B_x:\pi_1(\M_{1,1}^\an,x)\to\Aut(\cU^\rel)$$
induced from the de Rham and Betti connections are the same (up to an inner automorphism). This follows if the induced homomorphisms
$$\theta^\DR_x:\cG^\DR_x\to\Aut(\cU^\rel)\quad\text{and}\quad\theta^\B_x:\cG^\B_x\to\Aut(\cU^\rel)$$
are the same (up to an inner automorphism). The result follows by applying a rigidity lemma \cite[Lem. 14.1]{kzb} with $\Gamma=\SL_2\ltimes\cU^\rel$, $N=\mathcal{N}=\cU^\rel$, and $\phi=\theta^\B_x$.
\end{proof}


\part{Applications to Periods of Modular Forms}\label{pmf}

In this part, we discuss applications to periods of modular forms. In particular, we consider periods of the relative completion of $\SL_2(\Z)$, namely, multiple modular values defined by Brown \cite{brown:mmm}. They are given by (regularized\footnote{See \cite[\S 4]{brown:mmm} for the regularization process.}) iterated integrals\footnote{See \cite[Def. 5.7]{rc} for the definition of these generalized iterated integrals.} of {\em modular forms of the second kind}\footnote{See Def. \ref{mfsk} for the definition of modular forms of the second kind.}. We demonstrate with the simplest unfamiliar examples of multiple modular values, and provide explicit computations and constructions. 

For multiple modular values in length\footnote{Length refers to the number of forms one puts in an iterated integral.} one, we focus on periods of cusp forms. Each Hecke eigen cusp form $f$ is associated with a rank two cuspidal Hecke eigenspace, and should have four periods, corresponding to the entries of a $2\times2$ period matrix. Two of them are classical, given by Eichler integrals \cite{manin}. Critical values of the $L$-function $L(f,s)$ associated to $f$, depending on parity of $s$, are rational multiples of these two classical periods (cf. Deligne's conjecture). The other two periods are called quasi-periods \cite{brown:real, brown-hain}, which are not well-known. We consider the first interesting example of cusp form, Ramanujan's cusp form $\Delta$ of weight 12. Numerical results for its quasi-periods were obtained in \cite[\S 8]{brown:real} by integrating a weakly modular form found in \cite[\S 5]{brown-hain}. We use a modular form of the second kind to integrate, then numerically check that our method gives the same quasi-periods up to a linear combination. 

The main application of our $\Q$-de Rham theory is to construct iterated integrals of modular forms of the second kind. Brown \cite{brown:mmm}, generalizing Manin's work \cite{manin1,manin2}, mainly studied and computed those iterated integrals of modular forms that are `totally holomorphic', i.e. involving only {\em holomorphic modular forms}\footnote{Holomorphic modular forms are defined in Section \ref{cmf}.}. We show, by way of example, how to construct a double integral, i.e. iterated integral of length two, that involves a {\em non-holomorphic modular form}\footnote{Non-holomorphic modular forms refer to those modular forms of the second kind that are not holomorphic.}. The construction is explicit and essentially follows the procedure in Proposition \ref{procedure} up to length two. This is applicable to arbitrary length iterated integrals, from which one can carry out numerical computations to obtain all multiple modular values.  

\section{Periods and Quasi-periods of Cusp Forms}\label{pcf}
In the context of periods, the Eichler--Shimura isomorphisms \cite{eichler,shimura} (cf. (\ref{a})--(\ref{c}) in the introduction of Section \ref{cmfano}) can be expressed as
\begin{align}
M_{2n+2}\otimes_\Q\C&\xrightarrow{\sim} H^1(\M_{1,1}^\an,S^{2n}\H_\B)^+\otimes_\Q\C \\
S_{2n+2}\otimes_\Q\C&\xrightarrow{\sim} H^1(\M_{1,1}^\an,S^{2n}\H_\B)^-\otimes_\Q\C
\end{align}
where $S_{2n}\subset M_{2n}$ denote the subspace of cusp forms in the space of modular forms of weight $2n$ and level one, and $\pm$ denote eigenspaces with respect to the real Frobenius (complex conjugation). They are pieces of a canonical comparison isomorphism
\begin{equation}\label{ADR}
\comp_{\B,\DR}: H^1_\DR(\M_{1,1}, S^{2n}\cH)\otimes_\Q\C\xrightarrow{\sim} H^1(\M_{1,1}^\an,S^{2n}\H_\B)\otimes_\Q\C
\end{equation}
given by Theorem \ref{ADRstr}. This isomorphism is compatible with the action of Hecke operators, cf. Remark \ref{sQstr}. The Hecke eigenspace of an Eisenstein series has dimension one, and that of a cusp form has dimension two. 

Recall, from Theorem \ref{ADRstr}, that we have a basis for $H^1_\DR(\M_{1,1}, S^{2n}\cH)$ represented by modular forms of the second kind: $\wt\omega_f$ for each Hecke eigenform $f$ of weight $2n+2$, and $\wt\omega_{j,k}$ for each pair of positive integers $(j,k)$ such that $4j+6k=2n$.

For any Hecke eigenform $f$, the image of
$$\wt{\omega}_f=\begin{array}{|cc}
(\omega^{(0)}_f,\omega^{(1)}_f) &0\\
0& 0\\
\hline
\end{array}
$$
under the comparison isomorphism is given by the cohomology class of the cocycle:
\begin{equation}\label{ei}
\gamma\mapsto\int^{\tau_0}_{\gamma^{-1}\tau_0}\frac23(2\pi i)^{2n+1}f(\tau)(\tau\Y-\X)^{2n}d\tau
\end{equation}
where the integrand on the right hand side is the pullback $\psi^*\omega_f$ on the upper half plane of $\omega_f$ (cf. equation (\ref{omegaf})) along the map $\psi$ (see (\ref{psi})); and we have rewritten $\T=2\pi i(\tau\Y-\X)$ where $\X$, $\Y$ is a basis of $\H_\B$, which corresponds to the first Betti cohomology of the elliptic curve $\C/(\Z\oplus\Z\tau)$, cf. \cite[Appendix A]{brown-hain}. The cohomology class does not depend on the base point $\tau_0$. If $f$ is a cusp form, and $\tau_0$ is the cusp $i\infty$, then the right hand side becomes a classical Eichler integral \cite{eichler}. 

For each pair $(j,k)$, the image of 
$$\wt{\omega}_{j,k}=\begin{array}{|cc}
(\omega^{(0)}_{j,k},\omega^{(1)}_{j,k}) &0\\
0& l_{j,k}\\
\hline
\end{array}
$$
under the comparison isomorphism is given by the cohomology class of the cocycle:
\begin{equation}\label{eisk}
\gamma\mapsto\int^{\tau_0}_{\gamma^{-1}\tau_0}\psi^*\wt\omega_{j,k}.
\end{equation}
where we think of $\wt\omega_{j,k}$ geometrically as a 1-form (Remark \ref{ginter}), and is pulled back to a 1-form on the upper half plane via $\psi$; and we have rewritten $\Ss,\T\in\cH$ in terms of $\X,\Y\in\H_\B$ with
$$\Ss=\Y-2G_2(\tau)\T\qquad\text{and}\qquad\T=2\pi i(\tau\Y-\X),$$
cf. \cite[(2.12) \& Appendix A]{brown-hain}, here we use Zagier's normalization for the weight 2 Eisenstein series $\Eis_2(\tau)$ \cite{zagier}.

We now focus on the case of $n=5$ for (\ref{ADR}). Both sides then have dimension 3. The Hecke operators split it into two eigenspaces: one is 1-dimensional for the Eisenstein series $G_{12}(\tau)$, the other is 2-dimensional for the Ramanujan cusp form $\Delta(\tau)$.

\begin{example}{\bf A cocycle for Eisenstein series $G_{12}$ of weight 12.}
There is a canonical rational cocycle $e^0_{12}$ representing a cohomology class in $H^1(\M_{1,1}^\an,S^{2n}\H_\B)$ on the Betti side, computed explicitly in \cite[\S 7.4]{brown:mmm}. The image of $\wt\omega_{G_{12}}$ under the comparison isomorphism is the cohomology class $(2\pi i)^{11}[e^0_{12}]$. This cohomology class is represented by the cocycle
\begin{equation}\label{eco}
(2\pi i)^{11}e^0_{12}+\frac{10!}{2}\zeta(11)\delta^0(\Y^{10}),
\end{equation}
where $\zeta(\cdot)$ is the Riemann zeta function, and $\delta^0(\Y^{10})$ is a coboundary \cite[\S 2.3]{brown:mmm}. This cocycle was computed by Brown \cite[Lem. 7.1]{brown:mmm} from regularized Eichler integrals in (\ref{ei}), choosing the base point $\tau_0$ to be the unit tangent vector $\partial/\partial q$ at the cusp.
\end{example}

\begin{example}{\bf Periods of Ramanujan's cusp form $\Delta$ of weight 12.}
First, we compute classical periods of $\Delta$. They are totally holomorphic multiple modular values. We restrict the isomorphism (\ref{ADR}) to the 2-dimensional cuspidal Hecke eigenspace of $\Delta$. One then picks a basis $[P^+], [P^-]$ on the Betti side that is compatible with the real Frobenius as in \cite{brown:real}. Here $[P^+]$ and $[P^-]$ are cohomology classes represented respectively by cocycles $P^+$ and $P^-$ for $\SL_2(\Z)$ with coeffients in $S^{10}\H_\B$. Note that $S=\begin{pmatrix}0 & -1 \cr 1 & 0 \end{pmatrix}$ and $T=\begin{pmatrix}1 & 1 \cr 0 & 1\end{pmatrix}$ generate the group $\SL_2(\Z)$. Both $P^+$ and $P^-$ are cuspidal \cite[(7.3)]{brown:mmm}, i.e. their values on $T$ vanish. They are thus determined by their values on $S$:
$$P^+_S=\frac{36}{691}(\Y^{10}-\X^{10})+\X^2\Y^2(\X^2-\Y^2)^3\; , \; P^-_S=4\X^9\Y-25\X^7\Y^3+42\X^5\Y^5-25\X^3\Y^7+4\X\Y^9.$$ 
Using the method in \cite[\S 2]{brown:real}, picking a finite base point $\tau_0$ in the upper half plane, one computes the cocycle associated to $\wt\omega_\Delta$ in (\ref{ei}). This cocycle is, up to a coboundary,
$$\omega^+P^++i\omega^-P^-$$
where 
$$
\omega^+=-45944515.206396796502\ldots \quad , \quad
\omega^-=-3723343.5862069345779\ldots\quad
$$ 
are classical periods of $\Delta$. They agree with numerical values in the literature, cf. \cite{brown:real, kz}.

Now we compute quasi-periods\footnote{This follows the terminology for algebraic curves, where quasi-periods are periods of differential forms of the second kind.} of $\Delta$. We need to consider the modular form of the second kind $\wt\omega_{1,1}$. Using the same method above, one computes the cocycle in (\ref{eisk}) associated to $\wt\omega_{1,1}$. This cocycle is not a priori cuspidal, as it involves a multiple of the Eisenstein cocycle (\ref{eco}). After removing this multiple, the remaining cuspidal cocycle, up to a coboundary, is
$$\eta^+P^++i\eta^-P^-$$
where 
$$
\eta^+=-80.895549086122696192\ldots \quad , \quad
\eta^-=6.5585174555519361006\ldots\quad
$$
are quasi-periods of $\Delta$. They are the first examples of multiple modular values that are not totally holomorphic. One can check numerically that $\begin{pmatrix} \eta^+ \cr \eta^- \end{pmatrix}$ is a $\Q$-linear combination (see Remark \ref{lc} below) of the periods $\begin{pmatrix} \wt\eta_+ \cr \wt\eta_- \end{pmatrix}$ and $\begin{pmatrix} \wt\omega_+ \cr \wt\omega_- \end{pmatrix}$ computed by Brown \cite[\S 8.3]{brown:real}, and that the period matrix 
\begin{equation}\label{pm}
\begin{pmatrix} \eta^+ & \omega^+ \cr i\eta^- & i\omega^-\end{pmatrix}
\end{equation}
of $\Delta$ has determinant $-(2\pi i)^{11}$, which is compatible with \cite[Thm. 1.7]{brown-hain}.
\end{example}

\begin{remark}
Using the same method, one can compute quasi-periods of any Hecke eigen cusp form.
\end{remark}

\begin{remark}\label{lc}
Based on numerical computations, we have
$$\begin{pmatrix} \eta^+ \cr \eta^- \end{pmatrix}=-\frac{1}{10!}\left[\frac32\begin{pmatrix} \wt\eta_+ \cr \wt\eta_- \end{pmatrix}+\left(4\cdot 691+\frac{237}{691}\right)\begin{pmatrix} \wt\omega_+ \cr \wt\omega_- \end{pmatrix}\right].$$
It would be interesting to derive the linear combination using cohomological calculations. This would help describe the $\Q$-de Rham classes $[\wt\omega_{j,k}]$ represented by modular forms of the second kind, on which one could possibly define an action of Hecke operators.
\end{remark}

\section{Iterated Integrals of Modular Forms of the Second Kind}\label{dbint}
In this section, we construct iterated integrals of modular forms of the second kind, from which one can compute multiple modular values. Based on the de Rham theory of the relative completion \cite{rc}, one needs these iterated integrals to be closed, i.e. homotopy invariant\footnote{In \cite{rc}, these iterated integrals are called locally constant, since they can be viewed as locally constant functions on the path space.}, on $\M_{1,1}$. This requirement is already encoded in our construction of the connection form $\wt\Omega$ (Section \ref{coc}). From $\wt\Omega$, one can explicitly construct closed iterated integrals of modular forms of the second kind in arbitrary length. We show this by an example of a double integral that involves a non-holomorphic modular form. First, we review how to construct a closed double integral on a Riemann surface.

\begin{example}{\bf Double integrals of meromorphic 1-forms.}\label{recipe}
Suppose we are given two meromorphic 1-forms $\omega$ and $\eta$ on a Riemann surface $M$. To construct a closed iterated integral of these two differential forms, one follows the recipe in Hain \cite[\S 3]{gmhs}. Define 
$$\omega\cup\eta=\sum a_j p_j\qquad a_j\in\C,\ p_j\in M$$
if, locally, $\omega=dF$ and
$$(2\pi i)^{-1}\Res_{p_j}F\eta=a_j.$$
One can find another meromorphic 1-form $\xi$ such that 
$$(2\pi i)^{-1}\Res_{p_j}\xi=-a_j.$$
Then the double integral
$$\int\omega\eta+\xi$$
is a closed iterated integral on $M$.
\end{example}

\begin{remark}
If both $\omega$ and $\eta$ are holomorphic, then one could set $\xi=0$. This is the reason why `totally holomorphic' iterated integrals of modular forms are easy to construct; one essentially just strings together the holomorphic modular forms, see \cite[\S 3]{brown:mmm}, \cite{manin1,manin2}. 
\end{remark}

Recall that we have two classes $[\wt\omega_\Delta]$ and $[\wt\omega_{1,1}]$ in $H^1_\DR(\M_{1,1},S^{10}\cH)$, each can be represented by a modular form of the second kind. By Examples \ref{deltaform} \& \ref{cocycle}, the class $[\wt\omega_\Delta]$ is represented by 
$$\wt{\omega}_{\Delta}=\begin{array}{|cc}
(\omega_{\Delta}^{(0)},\omega_{\Delta}^{(1)}) &0\\
0& 0\\
\hline
\end{array}
$$
where $\omega_\Delta^{(j)}$ is the restriction to $U_j$ of a {\em global} 1-form on $\M_{1,1}$
$$\omega_\Delta=\alpha\T^{10}=(2udv-3vdu)\T^{10}.$$
By Example \ref{fndrc}, the class $[\wt\omega_{1,1}]$ is represented by 
$$\wt{\omega}_{1,1}=\begin{array}{|cc}
(\omega_{1,1}^{(0)},\omega_{1,1}^{(1)}) &0\\
0& l_{1,1}\\
\hline
\end{array}
$$
where $\omega_{1,1}^{(0)}=\frac{9\alpha}{u^2\Delta}\Ss^{10}\in\F_{10}^1(U_0)$, $\omega_{1,1}^{(1)}=-\frac{u\alpha}{2v^2\Delta}\Ss^{10}-\frac{5\alpha}{4v\Delta}\Ss^9\T\in\F_{10}^1(U_1)$, and $l_{1,1}=\frac 1{uv}\Ss^{10}\in\F_{10}^0(U_{01})$.
As in (\ref{dw}), we write
$$\wt\omega_\Delta=\wt\Omega_\Delta\otimes\T^{10}\quad\text{and}\quad\wt\omega_{1,1}=\wt\Omega_{1,1}^{10}\otimes\Ss^{10}+\wt\Omega_{1,1}^9\otimes\Ss^9\T,$$
where 
$\wt\Omega_\Delta=\begin{array}{|cc}
(\alpha,\alpha) &0\\
0& 0\\
\hline
\end{array}$,
$\wt\Omega_{1,1}^{10}=\begin{array}{|cc}
(\frac{9\alpha}{u^2\Delta},-\frac{u\alpha}{2v^2\Delta}) &0\\
0& \frac 1{uv}\\
\hline
\end{array}$,
and 
$\wt\Omega_{1,1}^{9}=\begin{array}{|cc}
(0,-\frac{5\alpha}{4v\Delta}) &0\\
0& 0\\
\hline
\end{array}$.

Our objective is to construct, from modular forms of the second kind $\wt\omega_\Delta$ and $\wt\omega_{1,1}$, a double integral
$$\int\wt\omega_\Delta\wt\omega_{1,1}+\wt\xi$$
which generalizes Example \ref{recipe}, and makes sense geometrically on $\M_{1,1}$.

Set a skew symmetric inner product on $\cH$ such that $\li\Ss,\T\ri=-\li\T,\Ss\ri=-1$, cf. \cite[\S 5.1]{brown-hain}. This induces a symmetric inner product $\li,\ri$ on $S^{2n}\cH$, with $\li\T^{2n},\Ss^{2n}\ri$=1. This inner product and the wedge product (\ref{wp}) on $K^\bullet(\M_{1,1})$ induce a pairing $\{,\}$ on modular forms of the second kind and a cup product $\cup$ on $H^1_\DR(\M_{1,1}, S^{2n}\cH)$, cf. \cite[\S 5.4--5.5]{brown-hain}. Now we are ready to construct our first new example of iterated integrals of modular forms. It is a closed double integral that involves a non-holomorphic modular form.

\begin{example}{\bf A double integral of modular forms of the second kind.}\label{di}
By the pairing just defined, we have 
$$
\{\wt\omega_\Delta,\wt\omega_{1,1}\}=(\wt\Omega_\Delta\wedge\wt\Omega_{1,1}^{10})\li\T^{10},\Ss^{10}\ri=\begin{array}{|cc}
(0,0) &0\\
0 & -\frac{\alpha}{uv}\\
0& 0\\
\hline
\end{array}
$$
where $-\frac{\alpha}{uv}=-\frac{2udv-3vdu}{uv}=3\frac{du}{u}-2\frac{dv}{v}$. One can find a 1-cochain $\wt\xi$
$$\wt\xi:=\begin{array}{|cc}
(\xi^{(0)},\xi^{(1)}) &0\\
0& l\\
\hline
\end{array}$$
with $\xi^{(0)}=3\frac{du}{u}$, $\xi^{(1)}=2\frac{dv}{v}$ and $l=0$, so that 
\begin{equation}\label{xi}
D\wt\xi+\{\wt\omega_\Delta,\wt\omega_{1,1}\}=0.
\end{equation}
This is essentially the procedure of Proposition \ref{procedure} in the case $n=2$. This provides us with, formally, a double integral on $\M_{1,1}$
$$\int\wt\omega_\Delta\wt\omega_{1,1}+\wt\xi:=\begin{array}{|cc}
(\int\omega^{(0)}\eta^{(0)}+\xi^{(0)},\int\omega^{(1)}\eta^{(1)}+\xi^{(1)}) &0\\
0& 0\\
\hline
\end{array}$$
It consists of a double integral
$$\int\omega^{(j)}\eta^{(j)}+\xi^{(j)}$$
on each open subset $U_j$ of $\M_{1,1}$ for $j=0,1$, with $\omega^{(0)}=\omega^{(1)}=\alpha$ coming from the top row of $\wt\Omega_\Delta=\begin{array}{|cc}
(\alpha,\alpha) &0\\
0& 0\\
\hline
\end{array}$, along with $\eta^{(0)}:=\frac{9\alpha}{u^2\Delta}$ and $\eta^{(1)}:=-\frac{u\alpha}{2v^2\Delta}$ coming from the top row of
$\wt\Omega_{1,1}^{10}=\begin{array}{|cc}
(\frac{9\alpha}{u^2\Delta},-\frac{u\alpha}{2v^2\Delta}) &0\\
0& \frac 1{uv}\\
\hline
\end{array}$.
In fact, both double integrals are closed. One can follow Example \ref{recipe} and check that on $U_0$
$$\omega^{(0)}\cup\eta^{(0)}=-3[\rho]$$
where $[\rho]\in\M_{1,1}^\an$ is defined by the $\G_m$-orbit $u=0$. Therefore, the 1-form $\xi^{(0)}=3\frac{du}{u}$ makes the double integral $\int\omega^{(0)}\eta^{(0)}+\xi^{(0)}$ closed. On $U_1$, similarly, one can check that
$$\omega^{(1)}\cup\eta^{(1)}=-2[i]$$
where $[i]\in\M_{1,1}^\an$ is defined by the $\G_m$-orbit $v=0$. Therefore, the 1-form $\xi^{(1)}=2\frac{dv}{v}$ makes the double integral $\int\omega^{(1)}\eta^{(1)}+\xi^{(1)}$ closed.
\end{example}

\begin{remark}
Note that $\frac{du}{u}=\frac 13 \frac{d\Delta}{\Delta}+\frac{27v\alpha}{u\Delta}$ and $\frac{dv}{v}=\frac 12 \frac{d\Delta}{\Delta}+\frac{u^2\alpha}{2v\Delta}$. So if one were to transfer either the residue $-3$ at $[\rho]$ from $\omega\cup\eta^{(0)}$ or the residue $-2$ at $[i]$ from $\omega\cup\eta^{(1)}$ to the cusp, one would get the same residue $-1$. This is expected since we are essentially computing a cup product of $[\omega_\Delta]$ and $[\wt\omega_{1,1}]$, with both $\eta^{(0)}$ and $\eta^{(1)}$ representing the same class $[\wt\omega_{1,1}]$ on opens $U_0$ and $U_1$ of $\M_{1,1}$. Moreover, the value $-1$ of the cup product is consistent with the determinant of the period matrix (\ref{pm}) of $\Delta$, cf. \cite[\S 5.5]{brown-hain}. 
\end{remark}

\begin{remark}
In general when constructing closed double integrals on $\M_{1,1}$, one can always find $\wt\xi$ satisfying equations like (\ref{xi}) since $H^2(K^\bullet(\M_{1,1}))$ vanishes. As for constructing iterated integrals of higher lengths, there is not much of a difference. One essentially follows the same steps when constructing the connection form $\wt\Omega$.
\end{remark}

We end with a remark on the periods given by Example \ref{di}.
\begin{remark}[\bf A motivic remark]
The double integrals in Example \ref{di} should produce periods of a biextension, cf. \cite[Example 13.4]{hdr}. The biextension data gives part of the period matrix of a mixed modular motive \cite{brown:mmm}. This motive has graded pieces that are pure motives isomorphic to $\Q$, $M$, $\wedge^2 M\cong\Q(-11)$, where $M$ is the motive associated to the Ramanujan cusp form $\Delta$. This is a typical example of a mixed modular motive, as its periods---given by iterated integrals involving modular forms of the second kind that are not holomorphic---are not `totally holomorphic' multiple modular values. In general, it is difficult to study these non-totally holomorphic multiple modular values. Since they are not canonically defined, and very few have been computed. 
\end{remark}

\end{document}